\documentclass{amsart}

\usepackage{amsmath}
\usepackage{amssymb}
\usepackage{mathrsfs}
\usepackage{enumerate}
\usepackage{color}

\newtheorem{thm}{Theorem}[section]

\newtheorem{cor}[thm]{Corollary}
\newtheorem{lem}[thm]{Lemma}

\newtheorem{defi}[thm]{Definition}

\newtheorem{rmk}[thm]{Remark}

\newcommand{\M}{{\mathcal M}}
\newcommand{\N}{{\mathcal N}}

\numberwithin{equation}{section}

\begin{document}

\title[\emph{}]{Johnson-Schechtman and Khinchine inequalities in noncommutative probability theory}

\author[]{Yong Jiao}
\address{School of Mathematics and Statistics, Central South University, Changsha, 410075, China}

\author[]{Fedor Sukochev}
\address{School of Mathematics and Statistics, University of NSW, Sydney,  2052, Australia}
\email{f.sukochev@unsw.edu.au}

\author[]{Dmitriy  Zanin}
\address{School of Mathematics and Statistics, University of NSW, Sydney,  2052, Australia}
\email{d.zanin@unsw.edu.au}

%\classno{Primary: 46L52, 46L53, 47A30; Secondary: 60G50}

% \keywords{Noncommutative probability, Independence, Symmetric operator space, Johnson-Schechtman inequality, Khinchine inequality}

\thanks{Yong Jiao is supported by NSFC(11471337), Hunan Provincial Natural Science Foundation(14JJ1004) and The International Postdoctoral Exchange Fellowship Program. Fedor Sukochev and Dmitriy Zanin are supported by the Australian Research Council.}
\maketitle

\begin{abstract} We prove disjointification inequalities due to Johnson and Schechtman for noncommutative random variables independent in the sense of Junge and Xu. In the same setting, we also prove noncommutative Khinchine inequalities. These inequalities are proved both for symmetric operator spaces and for modulars.
\end{abstract}

\section{Introduction}

The classical Khinchine inequality asserts that for every $0<p<\infty$ and for every finite sequence $\{\alpha_k\}_{k\geq0}\subset \mathbb C$
$$\|\sum_k\alpha_kr_k\|_p\approx_p \big(\sum_k|\alpha_k|^2\big)^{1/2},$$
where $\{r_k\}_{k\geq0}$ is a Rademacher sequence on a probability space. Here and in what follows, $A\approx_p B$ means that there are constants $C_p>0$ such that $A\leq C_p B$ and $B\leq C_p A.$  In a remarkable paper \cite{R}, Rosenthal generalised the Khinchine inequality by replacing $\{r_k\}_{k\geq0}$ with an arbitrary sequence $\{f_k\}_{k\geq0}\subset L_p(0,1),$ $p>2,$ of independent mean zero random variables. More precisely, Theorem 3 in \cite{R} states that
\begin{equation}\label{R}
\Big\|\sum_{k=0}^nf_k\Big\|_p\approx_{p}\Big(\sum_{k=0}^n\|f_k\|_p^p\Big)^{\frac1p}+\Big(\sum_{k=0}^n\|f_k\|_2^2\Big)^{\frac12}.
\end{equation}
Carothers and Dilworth \cite{CD1} extended the result above to the case of Lorentz spaces $L_{p,q}$, $1\leq p<\infty,\,0<q\leq\infty.$

The real breakthrough was achieved by Johnson and Schechtman \cite{JS} who established a far reaching generalisation of Rosenthal's result for symmetric Banach and quasi-Banach function spaces (see next section for precise definitions of these notions and subsequent terms and symbols). Let $E$ be a symmetric space on $[0,1]$ and  let $Z_E^1,Z_E^2$ be symmetric spaces on $(0,\infty)$ (defined in Subsection \ref{subsect ze}). It is established in \cite{JS} that
\begin{equation}\label{JS}
\Big\|\sum_{k=0}^nf_k\Big\|_E\approx_{E} \Big \|\bigoplus_{k=0}^nf_k\Big\|_{Z_E^2},\mbox{ respectively, }\Big\|\sum_{k=0}^nf_k\Big\|_E\approx_{E} \Big \|\bigoplus_{k=0}^nf_k\Big\|_{Z_E^1}
\end{equation}
for every sequence $\{f_k\}_{k=0}^n$ of independent mean zero (respectively, positive) random variables whenever $L_p\subset E$ for some $p<\infty.$ Here, $\bigoplus _{k=0}^nf_k:=\sum_{k=0}^nf_k(\cdot-k)\chi_{(k,k+1)}$ is a disjoint sum of independent mean zero random variables $\{f_k\}_{k=0}^n$ which is a Lebesgue measurable function on $(0,\infty)$. In the commutative case, the best possible results were achieved in \cite{AS-aop,pacific} by using the so-called Kruglov operator/property introduced in \cite{Br1994,AS2005}  (see detailed exposition of this theme in \cite{AS3}).

Junge, Parcet and Xu \cite{JPX} extended Rosenthal inequality \eqref{JS} to the realm of (noncommutative) free probability theory (we refer the reader to \cite{VDN} and \cite{NS1} for the necessary background) in the case that $E=L_p$ for $1\leq p<\infty,$ and for $p=\infty,$ we refer to Voiculescu \cite{V1}. Recently, a version of free Kruglov operator was constructed in \cite{SZfreeK}. By using a free Kruglov operator, a version of Johnson-Schechtman inequalities in the setting of free probability theory was obtained in \cite{SZfreeK}.

In 2008, Junge and Xu \cite{JX} (see also \cite{Sarym}) introduced the widest possible definition of independence for noncommutative random variables; see Definition \ref{NC defi} in Section 2. Their main result (Theorem 2.1 and Corollary 2.4 in \cite{JX}) reads as follows.

\begin{thm} Let $(\mathcal{M},\tau)$ be a noncommutative probability space. If $x_k\in L_p(\mathcal{M}),$ $k\geq0,$ $2\leq p<\infty,$ are mean zero independent random variables, then
\begin{equation}\label{RJX}
\Big\|\sum_{k=0}^nx_k\Big\|_p\approx_{p}\Big(\sum_{k=0}^n\|x_k\|_p^p\Big)^{\frac1p}+\Big(\sum_{k=0}^n\|x_k\|_2^2\Big)^{\frac12}.
\end{equation}
\end{thm}
In the same paper they also treated the case $1<p\leq 2$ (see \cite[Theorem 3.2]{JX}), however the right hand side becomes incomputable. Also the proofs in \cite[Theorem 3.2]{JX} are quite sketchy.

It is very natural to pursue a noncommutative version of the Johnson-Schechtman inequality \eqref{JS} for random variables independent in the sense of Junge-Xu \cite{JX}. This is one of the main aims in the present paper. Our first main result can be stated as follows. Its proof can be found in Section \ref{js section}. We refer to Section 2 for the unexplained notations.

\begin{thm}\label{js thm} Let $E=E(0,1)$ be a symmetric Banach function space and let $(\mathcal{M},\tau)$ be a noncommutative probability space. If $x_k\in E(\mathcal{M}),$ $k\geq0,$ are mean zero independent random variables and
\begin{enumerate}[{\rm (i)}]
\item\label{maina} if $E\in{\rm Int}(L_1,L_q),$ $1\leq q<\infty,$ then
$$\big\|\sum_{k\geq0}x_k\big\|_{E(\M)}\lesssim_E\big\|\sum_{k\geq0}x_k\otimes e_k\big\|_{Z_E^2(\M\bar{\otimes}\ell_\infty)}.$$
\item\label{mainb} if $E\in{\rm Int}(L_p,L_{\infty}),$ $1<p\leq\infty,$ then
$$\big\|\sum_{k\geq0}x_k\big\|_{E(\M)}\gtrsim_E\big\|\sum_{k\geq0}x_k\otimes e_k\big\|_{Z_E^2(\M\bar{\otimes}\ell_\infty)}.$$
\item\label{mainc} if $E\in{\rm Int}(L_p,L_q),$ $1<p\leq q<\infty,$ then
$$\big\|\sum_{k\geq0}x_k\big\|_{E(\M)}\approx_E\big\|\sum_{k\geq0}x_k\otimes e_k\big\|_{Z_E^2(\M\bar{\otimes}\ell_\infty)}.$$
\end{enumerate}
\end{thm}

At the time of writing, it is not clear to the authors whether the condition on $E$ in \eqref{mainc} is sharp.

Similarly to the second equality in \eqref{JS}, when $\{x_k\}_{k\geq0}$ is a sequence of independent positive random variables, we have the following assertion.

\begin{cor}\label{js positive} Let $E=E(0,1)$ be a symmetric Banach function space and let $(\mathcal{M},\tau)$ be a noncommutative probability space. Let $x_k\in E(\mathcal{M}),$ $k\geq0,$ be positive independent random variables.
\begin{enumerate}[{\rm (i)}]
\item\label{corb} If $E$ is an arbitrary symmetric space, then
$$\big\|\sum_{k\geq0}x_k\big\|_{E(\M)}\gtrsim_E\big\|\sum_{k\geq0}x_k\otimes e_k\big\|_{Z_E^1(\M\bar{\otimes}\ell_\infty)}.$$
\item\label{cora} If $E\in{\rm Int}(L_1,L_q),$ $1\leq q<\infty,$ then
$$\big\|\sum_{k\geq0}x_k\big\|_{E(\M)}\lesssim_E\big\|\sum_{k\geq0}x_k\otimes e_k\big\|_{Z_E^1(\M\bar{\otimes}\ell_\infty)}.$$
\item\label{corc} If $E\in{\rm Int}(L_1,L_q),$ $1\leq q<\infty,$ then
$$\big\|\sum_{k\geq0}x_k\big\|_{E(\M)}\approx_E\big\|\sum_{k\geq0}x_k\otimes e_k\big\|_{Z_E^1(\M\bar{\otimes}\ell_\infty)}.$$
\end{enumerate}
\end{cor}

Our second main result establishes Khinchine inequalities for noncommutative independent random variables. It significantly extends/strengthens a number of previously known results. We mention \cite[Theorem 3.1]{J-BG} and \cite[Corollary 4.18]{DPPS} as well as purely commutative Khinchine inequality \cite[Theorem 22]{pacific}, which are augmented by Theorem \ref{khinchine} below. Its proof can be found in Section \ref{kh section}.

\begin{thm}\label{khinchine} Let $E=E(0,1)$ be a symmetric Banach function space and let $(\mathcal{M},\tau)$ be a noncommutative probability space. Let $x_k\in E(\mathcal{M}),$ $k\geq0,$ be mean zero independent random variables. If $E\in{\rm Int}(L_p,L_q),$ $1<p\leq q<\infty,$ then
$$\big\|\sum_{k\geq0}x_k\big\|_{E(\M)}\approx_E\Big\|\Big(\sum_{k\geq0}x_k^2\Big)^{\frac12}\Big\|_{E(\M)}.$$
\end{thm}

As a note of caution to the reader, we would like to emphasise that the noncommutative Khinchine inequality stated below is substantially different from vector-valued versions of that inequality studied in \cite{LP,LPP}. In fact, the proof of Theorem \ref{khinchine} does not depend on the self-adjointness of $x_k.$ Verbatim repetition of the argument from Section \ref{kh section} yields the following assertion.

\begin{rmk} Let $\{x_k\}_{k\geq0}\subset E(\M)$ be a sequence of not necessarily self-adjoint operators. If $\{x_k\}_{k\geq0}$ are independent and mean zero, then we have
$$\|\sum_{k\geq0}x_k\|_{E(\mathcal{M})}\approx_E\|(\sum_{k\geq0}|x_k|^2)^{\frac12}\|_{E(\mathcal{M})}\approx_E\|(\sum_{k\geq0}|x_k^*|^2)^{\frac12}\|_{E(\mathcal{M})}.$$
Clearly, the last equality fails\footnote{Let $(\M,\tau)=(M_n(\mathbb{C}),\frac1n{\rm Tr}).$ Set $x_k=e_{k,0},$ $0\leq k<n.$ It is immediate that
$$\Big\|\Big(\sum_{k=0}^{n-1}|x_k|^2\Big)^{\frac12}\Big\|_{L_p(\mathcal{M})}=n^{1/2-1/p},\quad \Big\|\Big(\sum_{k=0}^{n-1}|x_k^*|^2\Big)^{\frac12}\Big\|_{L_p(\mathcal{M})}=1.$$
Of course, the sequence $\{x_k\}_{k\geq0}$ is not independent.} when the independence assumption is omitted.
\end{rmk}

Note that the assertions of Theorems \ref{js thm} and \ref{khinchine} hold also for arbitrary quasi-Banach symmetric spaces. Proofs require only minor adjustments. The assertions  of Corollary \ref{js positive} \eqref{cora},\eqref{corc} also hold for arbitrary quasi-Banach symmetric spaces. However, the assertion of Corollary \ref{js positive} \eqref{corb} requires additional restrictions on a quasi-Banach symmetric space $E$ (e.g. it suffices that the quasi-norm $\|\cdot\|_E$ is monotone with respect to Hardy-Littlewood submajorization).

We also prove $\Phi$-moment analogues of Johnson-Schechtman and Khinchine inequalities for noncommutative independent random variables. Let $\Phi$ be an Orlicz function on $[0,\infty)$. The study of $\Phi$-moment inequalities was initiated by Kruglov \cite{K} for classical independent random variables and by Burkholder-Gundy \cite{BG} for martingales. In recent years, $\Phi$-moment inequalities have been extended to noncommutative martingales \cite{BC,BCO,JSZZ} and freely independent random variables \cite{JSXZ}. In particular, in \cite{JSXZ}, the substantial difference\footnote{\cite[Theorem 1]{SZfreeK} shows that Johnson-Schechtman inequalities for freely independent random variables hold e.g. for every Orlicz space. \cite[Theorem 15]{JSXZ} establishes Johnson-Schechtman inequalities for freely independent random variables in the setting of modulars with $\Delta_2$-condition. \cite[Theorem 17]{JSXZ} demonstrates the necessity of $\Delta_2$-condition.} between Johnson-Schechtman inequalities and their $\Phi$-moment analogues has been demonstrated in the case of independent and freely independent random variables. Very recently, $\Phi$-moment analogues of the noncommutative Burkholder/Rosenthal inequalities are obtained in \cite{NW}. In what follows, we continue this line of research by providing a noncommutative $\Phi$-moment analogues of Theorem \ref{js thm} and Theorem \ref{khinchine}. The proof of the following theorem can be found in Section \ref{modular section}.

\begin{thm}\label{modular thm} Let $\Phi$ be an Orlicz function which is $p$-convex and $q$-concave, $1<p\leq q<\infty,$ and let $(\mathcal{M},\tau)$ be a noncommutative probability space. Let  $x_k\in L_{\Phi}(\M),$ $k\geq0,$ be independent random variables
\begin{enumerate}[{\rm (i)}]
\item\label{modpos} if $x_k,$ $k\geq0,$ are positive, then
$$\tau\Big(\Phi(|\sum_{k\geq0}x_k|)\Big)\approx_{\Phi}\mathbb E\Big(\Phi(\mu(X)\chi_{(0,1)})\Big)+\Phi(\|X\|_1),\quad X=\sum_{k\geq0}x_k\otimes e_k.$$
\item\label{modsym} if $x_k,$ $k\geq0,$ are mean zero, then
$$\tau\Big(\Phi(|\sum_{k\geq0}x_k|)\Big)\approx_{\Phi}\mathbb E\Big(\Phi(\mu(X)\chi_{(0,1)})\Big)+\Phi(\|X\|_{L_1+L_2}),\quad X=\sum_{k\geq0}x_k\otimes e_k.$$
\item\label{modkh} if $x_k,$ $k\geq0,$ are mean zero, then
$$\tau\Big(\Phi(|\sum_{k\geq0}x_k|)\Big)\approx_{\Phi}\tau\Big(\Phi\Big(\big(\sum_{k\geq0}x_k^2\big)^{\frac12}\Big)\Big).$$
\end{enumerate}
\end{thm}

We remark that Theorem \ref{modular thm} (\ref{modpos}) and (\ref{modsym}) improve/strengthen \cite[Corollary 5.2 and Theorem 5.6]{NW} in the case that conditional expectation used in \cite{NW} is a trace, and Theorem \ref{modular thm} (\ref{modkh}) strengthens \cite[Theorem 5.1]{BC}. For example, consider the Orlicz function $\Phi(t)=t^p\log(1+t^q)$ with $p>1$ and $q>0.$ It is easy to check that $\Phi$ is $p$-convex and $(p+q)$-concave. If, for example, $p=2$ or $p+q=2,$ then the results from \cite{NW} and \cite{BC} are not applicable, whereas Theorem \ref{modular thm} still holds.

\section{Preliminaries}\label{sect prel}

\subsection{Noncommutative symmetric spaces}
Let $\mathcal{M}$ be a finite (or semifinite) von Neumann algebra equipped with a faithful normal finite (or semifinite) trace $\tau$. We denote by $L_0(\mathcal{M},\tau)$, or simply $L_0(\mathcal{M})$, the family of all $\tau$-measurable operators \cite{FK}. Recall that $e_{(s,\infty)}(|x|)$ is the spectral projection of $x\in L_0(\mathcal{M})$ associated with the interval $(s,\infty)$. For $x\in L_0(\mathcal{M})$, the generalized singular value function is defined by
$$\mu(t,x)=\inf\{s>0: \tau(e_{(s,\infty)}(|x|))\leq t\}, \quad t>0.$$
The function $t\mapsto\mu(t,x)$ is decreasing and right-continuous \cite{FK}. For the case that $\M$ is the abelian von Neumann algebra $L_\infty(0,1)$ with the trace given by integration with respect to the Lebesgue measure,  $L_0(\mathcal{M})$ is the space of all measurable functions and $\mu(f)$ is the decreasing rearrangement of the measurable function $f$; see \cite{KPS,LT}.

A Banach (or quasi-Banach) function space $(E,\|\cdot\|_E)$ on the interval $(0,\alpha)$, $0<\alpha\leq \infty$ is called symmetric if for every $g\in E$ and for every measurable function $f$ with $\mu(f)\leq\mu(g)$, we have $f\in E$ and $\|f\|_E\leq\|g\|_E.$

For a given symmetric Banach (or quasi-Banach) space $(E,\|\cdot\|_E)$, we define the corresponding noncommutative space on $(\M,\tau)$ by setting \cite{KS}
\begin{equation}\label{ks def}
E(\M,\tau):=\{x\in L_0(\M,\tau):\mu(x)\in E\}.
\end{equation}
Endowed with the quasi-norm $\|x\|_{E(\M)}:=\|\mu(x)\|_E,$ the space $E(\M,\tau)$ (or, briefly, $E(\M)$) is called the noncommutative symmetric space associated with $(\M,\tau)$ corresponding to the function space $(E,\|\cdot\|_E)$. It is shown in \cite{S} that the quasi-norm space $(E(\M),\|\cdot\|_{E(\M)})$ is complete if $(E,\|\cdot\|_E)$ is complete.

Throughout this paper, we write $A\lesssim_{E} B$ if there is a constant $C_E>0$ depending only on $E$ such that $A\leq C_E B$; and we write $A\approx_E B$ if both $A\leq C_E B$ and $B\leq C_E A$ hold, possibly with different constants; we write $A\approx B$ if the inequalities above hold for an absolute constant $C$ which is independent of $E.$

\subsection{$Z_E^p$ spaces and interpolation spaces}\label{subsect ze}

As usual, $(L_p\cap L_q)(0,\infty)$ is the intersection of the Banach spaces $L_p(0,\infty)$ and $L_q(0,\infty).$ Its norm is given by the formula
$$\|f\|_{L_p\cap L_q}=\max\{\|f\|_p,\|f\|_q\}.$$
Also, $(L_p+L_q)(0,\infty)$ is the sum of the Banach spaces $L_p(0,\infty)$ and $L_q(0,\infty).$ Its norm is given by the formula
$$\|f\|_{L_p+L_q}=\inf\{\|g\|_p+\|h\|_q:\ f=g+h\}.$$
In the computations below, we often use the well-known Holmstedt formula (see \cite[Theorem 4.1]{H}), that is, for $0<p<q\leq\infty$
\begin{equation} \label{p+q}
\|f\|_{L_p+L_q} \approx \Big(\int_0^1\mu(t,f)^pdt\Big)^{\frac{1}{p}}+\Big(\int_1^\infty \mu(t,f)^qdt\Big)^{\frac{1}{q}},\quad f\in (L_p+L_q)(0,\infty).
\end{equation}
For every semifinite von Neumann algebra $\N$, the spaces $(L_p\cap L_q)(\N)$ and $(L_p+L_q)(\N)$ are defined according to \eqref{ks def} (see \cite{KS}).
For every $1\leq p,q<\infty,$ we have (see e.g. \cite[page 64]{PWS})
\begin{equation}\label{intersect banach dual}
((L_p\cap L_q)(\N))^*=(L_{\frac{p}{p-1}}+L_{\frac{q}{q-1}})(\N),
\end{equation}
\begin{equation}\label{sum banach dual}
((L_p+L_q)(\N))^*=(L_{\frac{p}{p-1}}\cap L_{\frac{q}{q-1}})(\N).
\end{equation}

The following useful construction can be found in \cite{JS}. If $E$ is a symmetric quasi-Banach space on the interval $(0,1)$, then the space $Z_E^p,$ $1\leq p\leq\infty,$ consists of all measurable functions $f\in L_1(0,\infty)+L_\infty(0,\infty)$ for which
$$\|f\|_{Z_E^p}:=\|\mu(f)\chi_{(0,1)}\|_E+\|f\|_{L_1+L_p}<\infty.$$
Clearly, $Z_E^p$ is a symmetric Banach function space on the semi-axis $(0,\infty).$ For every semifinite von Neumann algebra $\N$, the space $Z_E^p(\N)$ is defined according to \eqref{ks def} (see \cite{KS}).

Further, it is useful to note that $Z_{L_p}^2=(L_p+L_2)(0,\infty)$ for $p<2$ and $Z_{L_p}^2=(L_p\cap L_2)(0,\infty)$ for $p>2.$

Let $F_1,F_2$ and $F_1+F_2\subset E\subset F_1+F_2$ be symmetric quasi-Banach spaces. We say that $E$ is an interpolation space between $F_1$ and $F_2$ (written $E\in {\rm Int}(F_1,F_2)$) if, for every linear map $T:F_1+F_2\to F_1+F_2$ such that (the restrictions) $T:F_1\to F_1$ and $T:F_2\to F_2$ are bounded, we also have $T:E\to E$ is bounded. We refer to \cite{KPS} and \cite{KM} for some necessary background on interpolation.

In what follows, instead of writing $T_1:F_1\to E_1,$ $T_2:F_2\to E_2$ and $T_1|_{F_1\cap F_2}=T_2|_{F_1\cap F_2}$ we simply write $T:F_1\to E_1$ and $T:F_2\to E_2$ where $T$ is a common notation for $T_1$ and $T_2.$ 

\subsection{Noncommutative probability space and independence}

In this subsection we introduce some basic definitions concerning noncommutative probability spaces and noncommutative independence. Let $\mathcal{M}$ be a finite von Neumann algebra equipped with a  faithful normal trace $\tau$. If $\tau(1)=1$, we call $(\mathcal{M},\tau)$ a noncommutative probability space. Let $L_0^h(\mathcal{M},\tau)$ denote the set of all self-adjoint elements in $L_0(\mathcal{M},\tau),$ which are called (noncommutative) random variables.

We now introduce the noncommutative independence. The following definition of independence can be found in \cite[page 233]{JX}. In a sense, it is the widest possible definition of independence. The germ of this definition was known much earlier, see e.g. Sarymsakov's book \cite{Sarym}.
\begin{defi}\label{NC defi} Let $(\mathcal{M},\tau)$ be a noncommutative probability space.
\begin{enumerate}[{\rm (i)}]
\item We say that a sequence $\{\M_k\}_{k\geq0}$ of von Neumann subalgebras in $\M$ is independent (with respect to $\tau$) if for every $k$, $\tau(xy)=\tau(x)\tau(y)$ holds for all $x\in\mathcal{M}_k$ and for every $y$ in the von Neumann algebra generated by $\{\mathcal{M}_j\}_{j\neq k}.$
\item A sequence $\{x_k\}_{k\geq0}\subset L_0^h(\M)$ is said to be independent with respect to $\tau$ if the unital von Neumann subalgebras $\M_k,$ $k\geq0,$ generated by $x_k,$ $k\geq0,$ are independent.
\end{enumerate}
\end{defi}

Here, we follow a physical tradition (and a widely spread mathematical tradition in probability theory) that random variables are self-adjoint. However,  in \cite{JX} this assumption was not imposed and exactly the same definition of independence is given for arbitrary elements.

Examples 1.4 and 1.5 in \cite{JX} show that the notion of classical independence and that of free independence are very special cases of Definition \ref{NC defi}. Also, fermions (i.e. anti-commuting Pauli matrices) generate independent algebras (even though these algebras are very thin). However, \cite{JX} supplies other interesting examples of independent random variables. Namely, Example 1.6 in \cite{JX} demonstrates that $q$-independence introduced in \cite{BozSp} is also a special case of Junge-Xu independence.

\begin{rmk} Let $\N$ and $\M_k,$ $k\geq0,$ be von Neumann subalgebras of $\M.$ In \cite{JX}, the subalgebras $\{\M_k\}_{k\geq0}$ are said to be independent with respect to $\N$ if for every $k$, $\mathcal{E}_\N(xy)=\mathcal E_\N(x)\mathcal{E}_\N(y)$ holds for all $x\in\mathcal{M}_k$ and for every $y$ in the unital von Neumann algebra generated by $\{\mathcal{M}_j\}_{j\neq k}.$ This notion generalises the classical notion of {\it conditional independence}. Our Definition \ref{NC defi} is the special case of that in \cite{JX} with $\N=\mathbb C.$ That is, the notion introduced in Definition \ref{NC defi} generalises the classical notion of independence and coincides with the latter when $\M=L_{\infty}(0,1).$
\end{rmk}

\subsection{Orlicz functions and Orlicz spaces} Let $\Phi:\mathbb{R}\to\mathbb{R}_+$ be an Orlicz function, that is, $\Phi$ is an even convex function such that $\Phi(0)=0$ and $\Phi(\infty)=\infty.$ Given an Orlicz function $\Phi$ and an operator $x\in L_0(\M,\tau)$, we have by \cite[Corollary 2.8]{FK},
$$\tau(\Phi(|x|))=\int_0^\infty\Phi\big(\mu(t,x)\big)dt.$$
Obviously, if $\Phi(t)=t^p$ for $1\leq p<\infty,$ then this reduces to the usual $p$-moment of $|x|.$ By induction, it follows from \cite[Proposition 4.2 (ii)]{FK} that for every sequence $\{x_i\}_{i=0}^n\subset L_0(\M,\tau)$ and scalars $\{\lambda_i\}_{i=0}^n\in (0,1)$ with $\sum_{i=0}^n\lambda_i\leq1,$
\begin{equation}\label{convex phi}
\tau\Big(\Phi\Big(|\sum_{i=0}^n\lambda_ix_i|\Big)\Big)\leq \sum_{i=0}^n\lambda_i\tau\Big(\Phi(|x_i|)\Big).
\end{equation}
We now recall the definition of Orlicz spaces. Given an Orlicz function $\Phi$, the Orlicz function space $L_\Phi(0,\alpha)$, $0<\alpha\leq\infty$ is the set of all measurable functions $f$ on $(0,\alpha)$ such that
$$\|f\|_{L_\Phi}:=\inf\Big\{\lambda>0,\int_0^\infty\Phi\Big(\frac{|f(t)|}\lambda\Big)\Big\}\leq1.$$
An Orlicz function is said to satisfy $\Delta_2$-condition if there exists a constant $a>1$ such that $\Phi(at)\lesssim_a \Phi(t)$ for all $t>0.$ Note that if $\Phi$ satisfies $\Delta_2$-condition, then $f\in L_\Phi(0,\alpha)$ if and only if $\int_0^\infty\Phi(|f|)<\infty.$
If $\Phi$ satisfy the $\Delta_2$-condition, then by \eqref{convex phi} we have
\begin{equation}\label{phi quasi-trangle}
\tau(\Phi(|x+y|))\lesssim_\Phi\tau(\Phi(|x|))+\tau(\Phi(|y|)),\quad x,y \in L_\Phi(\M).
\end{equation}

Given $1\leq p\leq q\leq\infty$, an Orlicz function $\Phi$ is said to be $p$-convex if the function $t\to\Phi(t^{\frac 1p}),$ $t>0$ is convex, and $\Phi$ is said to be $q$-concave if the function $t\to\Phi(t^{\frac1q}),$ $t>0$ is concave. Orlicz function $\Phi$ satisfies $\Delta_2$-condition if and only if it is $q$-concave for some $q<\infty.$ (see e.g. \cite[Lemma 5]{AS2014}).

\subsection{Hardy-Littlewood and uniform submajorization}

In this subsection, we introduce the Hardy-Littlewood submajorization and uniform submajorization, which are used in the sequel.

\begin{defi}\label{HLM} Let $\M$ be a semifinite von Neumann algebra and $x,y\in (L_1+L_{\infty})(\M).$ We write $y\prec\prec x$ if
$$\int_0^t\mu(s,y)ds\leq\int_0^t\mu(s,x)ds,\quad t>0.$$
\end{defi}

It follows from \cite[Theorem 11]{DSZ} that for every Orlicz function $\Phi$
\begin{equation}\label{majorization phi}
x\prec\prec y\Longrightarrow {\tau}(\Phi(x))\leq {\tau}(\Phi(y)),\quad \forall x,y\in(L_1+L_{\infty})(\M).
\end{equation}

The following definition is introduced originally in \cite{KS} (see also \cite{LSZ} for a more readable exposition).
\begin{defi} Let $\M$ be a semifinite von Neumann algebra and let $x,y\in(L_1+L_{\infty})(\M).$ We say that $y$ is uniformly submajorized by $x$ (written $y\vartriangleleft x$) if there exists $\lambda \in \mathbb N$ such that
$$\int_{\lambda a}^b\mu(s,y)ds\leq\int_a^b\mu(s,x)ds,\quad\lambda a\leq b.$$
\end{defi}
Uniform submajorization is stronger than the Hardy-Littlewood submajorization. That is, $y\vartriangleleft x$ implies that  $y\prec\prec x.$

\subsection{Conditional expectations}

Let $(\mathcal{M},\tau)$ be a noncommutative probability space. If $\mathcal{D}$ is a von Neumann subalgebra of $\mathcal{M},$ then there exists a unique linear operator $\mathcal{E}_{\mathcal{D}}:\mathcal{M}\to\mathcal{D}$ such that, for all $A\in\mathcal{M}$ and $B\in\mathcal{D}$,
\begin{enumerate}[{\rm (a)}]
\item $\mathcal{E}_{\mathcal{D}}(AB)=\mathcal{E}_{\mathcal{D}}(A)B$
\item $\mathcal{E}_{\mathcal{D}}(BA)=B\mathcal{E}_{\mathcal{D}}(A)$
\item $\tau(\mathcal{E}_{\mathcal{D}}(A))=\tau(A)$.
\end{enumerate}
Furthermore, $\mathcal{E}_{\mathcal{D}}$ is positive unital trace preserving linear map. That is, $\mathcal{E}_{\mathcal{D}}$ is contraction from $\mathcal{M}$ onto $\mathcal{D}$ and it uniquely extends to a contraction from $L_1(\mathcal{M})$ onto $L_1(\mathcal{D}).$ In particular, $\mathcal{E}_{\mathcal{D}}(A)\prec\prec A.$ See, for example, \cite{SS} for these and other facts.

\subsection{Various tensor product algebras}

In what follows, $\ell_{\infty}$ is the von Neumann algebra of all bounded sequences (enumeration starts from $0$). It can be viewed as acting on the Hilbert space $\ell_2$ of all square-summable sequences. In what follows, $\{e_k\}_{k\geq0}$ are the standard unital vectors in $\ell_{\infty}.$ Also, $\mathcal{L}(\ell_2)$ is the von Neumann algebra of all bounded operators on $\ell_2.$ As usual, $\{e_{k,l}\}_{k,l\geq0}$ are the matrix units in $\mathcal{L}(\ell_2).$

Let $(\mathcal{M},\tau)$ be a noncommutative probability space. In what follows, we frequently use algebras $\mathcal{M}\bar{\otimes}\mathcal{L}(\ell_2),$ $\mathcal{M}\bar{\otimes}\ell_{\infty}$ and $\mathcal{M}\bar{\otimes}L_{\infty}(0,1).$ We equip those algebras with their natural (semi-)finite traces. i.e. $\tau\otimes{\rm Tr},$ $\tau\otimes\Sigma$ and $\tau\otimes\int.$

Here, ${\rm Tr}$ is the standard trace on $\mathcal{L}(\ell_2)$ given by the formula
$${\rm Tr}(A)=\sum_{k\geq0}A_{k,k},\quad A=\{A_{k,l}\}_{k,l\geq0}\in L_1(\mathcal{L}(\ell_2)).$$
$\Sigma:L_1(\ell_{\infty})=\ell_1\to\mathbb{C}$ is defined by the formula
$$\Sigma(\{x_k\}_{k\geq0})=\sum_{k\geq0}x_k,\quad x=\{x_k\}_{k\geq0}\in L_1(\ell_{\infty}).$$
And, $\int$ is just the Lebesgue integration on $L_{\infty}(0,1).$

Those operators affiliated to $\mathcal{M}\bar{\otimes}\ell_{\infty},$ which happen to be $\tau\otimes\Sigma-$measurable, are written as $x=\sum_{k\geq0}x_k\otimes e_k.$ Also,  $\tau\otimes{\rm Tr}-$measurable operators affiliated to $\mathcal{M}\bar{\otimes}\mathcal{L}(\ell_2)$ are written as $x=\sum_{k,l\geq0}x_{k,l}\otimes e_{k,l}.$

\section{Interpolation lemmas}\label{interpol section}

In this section, we provide several interpolation lemmas. Lemma \ref{first interpolation lemma} and Lemma \ref{second interpolation lemma} are used to prove Theorem \ref{js thm}, and Lemma \ref{third interpolation lemma} and Lemma \ref{fourth interpolation lemma} are employed for the proof of Theorem \ref{khinchine}.

Prior to stating lemmas, note that the space $Z_E^2(\mathcal{N})$ sits inside the space $(L_1+L_2)(\mathcal{N})$ and therefore, one can speak of the reduction of the operator $V$ (defined on the latter space) to $Z_E^2(\mathcal{N}).$

\begin{lem}\label{first interpolation lemma} Let $(\mathcal{M},\tau)$ be a finite von Neumann algebra and let $(\mathcal{N},\nu)$ be a semifinite atomless one. If $V:(L_1+L_2)(\mathcal{N})\to L_1(\mathcal{M})$ and $V:(L_q\cap L_2)(\mathcal{N})\to L_q(\mathcal{M}),$ $2\leq q<\infty,$ are bounded linear operators, then for every symmetric function space $E\in {\rm Int}(L_1,L_q),$ we have
$$\|V(x)\|_{E(\mathcal{M})}\lesssim_E \|x\|_{Z_E^2(\mathcal{N})},\quad x\in Z_E^2(\mathcal{N}).$$
\end{lem}
\begin{proof} Fix $x=x^*\in Z_E^2(\mathcal{N})$ and consider the spectral projections associated with the sets $(\mu(1,x),\infty)$ and $\{\mu(1,x)\},$ respectively,
$${\bf p}:=e_{(\mu(1,x),\infty)}(|x|)\quad {\rm and} \quad {\bf r}:=e_{\{\mu(1,x)\}}(|x|).$$
It follows immediately from the definition of singular value function that $\nu(\bf p)\leq1.$ Since $\N$ is atomless, there exists a projection ${\bf r_0}\leq{\bf r}$ such that $\nu(\bf r_0)=1-\nu({\bf p}).$ Let ${\bf q}={\bf p}+{\bf r_0}.$ Then $\nu({\bf q})=1$, $x{\bf q}={\bf q}x$ and $|x|{\bf q}\geq\mu(1,x){\bf q}.$ Clearly, $(L_1+L_2)({\bf q}\mathcal{N}{\bf q})=L_1({\bf q}\mathcal{N}{\bf q})$ and $(L_q\cap L_2)({\bf q}\mathcal{N}{\bf q})=L_q({\bf q}\mathcal{N}{\bf q}).$ By the assumption, the operator $V:L_1({\bf q}\mathcal{N}{\bf q})\to L_1(\mathcal{M})$ is bounded and also $V:L_q({\bf q}\mathcal{N}{\bf q})\to L_q(\mathcal{M})$ is bounded. Since $E\in {\rm Int}(L_1,L_q),$ it follows from \cite[Theorem 3.2]{DDP-Int} that $V:E({\bf q}\mathcal{N}{\bf q})\to E(\mathcal{M})$ is bounded. Note that $\mu(t,{\bf q}x{\bf q})=0$ for all $t\geq \nu(\bf q)=1$, we have
$$\mu({\bf q}x{\bf q})=\mu({\bf q}x{\bf q})\chi_{(0,1)}=\mu(x)\chi_{(0,1)}.$$
Therefore,
$$\|V({\bf q}x{\bf q})\|_{E(\mathcal{M})}\lesssim_E\|{\bf q}x{\bf q}\|_{E({\bf q}\mathcal{N}{\bf q})}=\|\mu(x)\chi_{(0,1)}\|_E.$$
On the other hand, $q$ commutes with $x$ and, therefore,
$$\mu(t,(1-{\bf q})x(1-{\bf q}))=\mu(t+1,x),\quad t>0.$$
Therefore, by \eqref{p+q}, we have
\begin{eqnarray*}
\|(1-{\bf q})x(1-{\bf q})\|_{L_q(\N)}&\leq&\|(1-{\bf q})x\|_{L_q(\N)}\leq\Big(\int_1^{\infty}\mu^q(t,x)dt\Big)^{1/q}\\&\leq&\|\mu(x)\chi_{(1,\infty)}\|_2+\mu(1,x)\lesssim\|\mu(x)\|_{L_1+L_2}.
\end{eqnarray*}
Similarly,
$$\|(1-{\bf q})x(1-{\bf q})\|_{L_2(\N)}\lesssim\|\mu(x)\|_{L_1+L_2}.$$
The above arguments guarantee that the element $(1-{\bf q})x(1-{\bf q})$ belongs to $(L_q\cap L_2)(\mathcal{N})$ and that
$$\|(1-{\bf q})x(1-{\bf q})\|_{(L_q\cap L_2)(\N)}\lesssim\|\mu(x)\|_{L_1+L_2}.$$
By assumption, we know that $V((1-{\bf q})x(1-{\bf q}))$ belongs to $L_q(\M).$ Consequently, taking into account that $L_q(\M)\subset E(\M)$ we obtain that
\begin{eqnarray*}
\|V((1-{\bf q})x(1-{\bf q}))\|_{E(\M)}&\lesssim_E&\|V((1-{\bf q})x(1-{\bf q}))\|_{L_q(\M)}\\
&\lesssim_E&\|(1-{\bf q})x(1-{\bf q})\|_{(L_q\cap L_2)(\N)}\\
&\lesssim&\|\mu(x)\|_{L_1+L_2}.
\end{eqnarray*}
Noting that $x{\bf q}={\bf q}x$, we arrive at
\begin{eqnarray*}
\|Vx\|_{E(\M)}&=&\|V({\bf q}x{\bf q})+V((1-{\bf q})x(1-{\bf q}))\|_{E(\M)}\\
&\leq&\|V({\bf q}x{\bf q})\|_{E(\M)}+\|V((1-{\bf q})x(1-{\bf q}))\|_{E(\M)}\\
&\lesssim_E&\|\mu(x)\chi_{(0,1)}\|_E+\|\mu(x)\|_{L_1+L_2}\lesssim\|x\|_{Z_E^2(\N)}.
\end{eqnarray*}
This proves the assertion for self-adjoint $x.$

Let now $x$ be an arbitrary element from $Z_E^2(\N).$ By splitting $x$ into its real part and imaginary parts, we obtain
$$\|Vx\|_{E(\M)}\leq \|V(\Re(x))\|_{E(\M)}+\|V(\Im(x))\|_{E(\M)}\lesssim_E\|x\|_{Z_E^2(\N)},\quad \forall x\in Z_E^2(\N).$$
\end{proof}

\begin{lem}\label{second interpolation lemma} Let $(\mathcal{M},\tau)$ be a finite von Neumann algebra and let $(\mathcal{N},\nu)$ be a semifinite atomless one. If $V:L_p(\mathcal{M})\to(L_p+L_2)(\mathcal{N}),$ $1\leq p\leq 2,$ and $V:L_{\infty}(\mathcal{M})\to(L_2\cap L_{\infty})(\mathcal{N}),$ are bounded linear operators, then for every $E\in{\rm Int}(L_p,L_\infty),$
$$\|V(x)\|_{Z_E^2(\mathcal{N})}\lesssim_E \|x\|_{E(\mathcal{M})},\quad x\in E(\mathcal{M}).$$
\end{lem}
\begin{proof} Without loss of generality, we may assume that $V$ maps self-adjoint operators to self-adjoint ones. Indeed, we can write $V=V_1+iV_2$, where
$$V_1:x\to \frac12\Big(V(x)+V(x^*)^*\Big)\quad {\rm and}\quad V_2:x\to \frac{1}{2i}\Big(V(x)-V(x^*)^*\Big),$$
so that both $V_1$ and $V_2$ map self-adjoint operators to self-adjoint ones.

Fix an element $x=x^*\in E(\mathcal{M})$ and consider the spectral projection of $|Vx|$ associated with the interval $(\mu(1,Vx),\infty),$
$${\bf p}:=e_{(\mu(1,Vx),\infty)}(|Vx|).$$
It is clear that $\nu({\bf p})\leq1.$ Fix a projection ${\bf q}\geq {\bf p}$ such that $\nu({\bf q})=1.$ Consider the operator
$$W:z\to {\bf q}\cdot Vz\cdot {\bf q},\quad z\in (L_p+L_{\infty})(\mathcal{M}).$$
By the assumption, $W:L_p(\mathcal{M})\to L_p({\bf q}\mathcal{N}{\bf q})$ and $W:L_{\infty}(\mathcal{M})\to L_{\infty}({\bf q}\mathcal{N}{\bf q})$ are contractions. Since $E$ is an interpolation space between $L_p(0,1)$ and $L_{\infty}(0,1),$ it follows from \cite[Theorem 3.2]{DDP-Int} that $W:E(\mathcal{M})\to E({\bf q}\mathcal{N}{\bf q})$ is a bounded map. Noting that $\mu({\bf q}\cdot Vx\cdot {\bf q})$ lives on the interval $(0,1),$ we have
\begin{equation}\label{1}
\|{\bf q}\cdot Vx\cdot {\bf q}\|_{Z_E^2(\mathcal{N})}=\|Wx\|_{E({\bf q}\mathcal{N}{\bf q})}\lesssim_E\|x\|_{E(\mathcal{M})}.
\end{equation}
By \eqref{p+q}, it is clear that
\begin{equation}\label{2}
\|(1-{\bf q})\cdot Vx\|_{L_{\infty}(\mathcal{N})}\leq \|(1-{\bf p})\cdot Vx\|_{L_{\infty}(\mathcal{N})}\leq\mu(1,Vx)\leq\|Vx\|_{(L_p+L_2)(\mathcal{N})},
\end{equation}
and since $1\leq p\leq 2$ (again by \eqref{p+q}),
\begin{equation}\label{3}
\|(1-{\bf q})\cdot Vx\|_{L_2(\mathcal{N})}\leq\|\mu(Vx)\chi_{(1,\infty)}\|_2\leq\|Vx\|_{(L_p+L_2)(\mathcal{N})}.
\end{equation}
Noting that $(L_2\cap L_{\infty})(\mathcal{N})\subset Z_E^2(\mathcal{N})$ and combining \eqref{1}, \eqref{2} and \eqref{3}, we obtain that
\begin{eqnarray*}
\|Vx\|_{Z_E^2(\mathcal{N})}&\leq&\|{\bf q}\cdot Vx\cdot {\bf q}\|_{Z_E^2(\mathcal{N})}+\|{\bf q}\cdot Vx\cdot (1-{\bf q})\|_{Z_E^2(\mathcal{N})}+\|(1-{\bf q})\cdot Vx\|_{Z_E^2(\mathcal{N})}\\
&\lesssim_E&\|x\|_{E(\mathcal{M})}+\|{\bf q}\cdot Vx\cdot (1-{\bf q})\|_{(L_2\cap L_{\infty})(\mathcal{N})}+\|(1-{\bf q})\cdot Vx\|_{(L_2\cap L_{\infty})(\mathcal{N})}\\
&\leq&\|x\|_{E(\mathcal{M})}+2\Big(\|(1-{\bf q})\cdot Vx\|_{L_2(\N)}+\|(1-{\bf q})\cdot Vx\|_{L_\infty(\N)}\Big)\\
&\leq&\|x\|_{E(\mathcal{M})}+4\|Vx\|_{(L_p+L_2)(\mathcal{N})}\lesssim\|x\|_{E(\mathcal{M})}+\|x\|_{L_p(\mathcal{M})}\\
&\lesssim& \|x\|_{E(\mathcal{M})}.
\end{eqnarray*}
This proves the assertion for self-adjoint $x.$

Let now $x$ be an arbitrary element from $E(\M).$ By splitting $x$ into its real part and imaginary parts and using the observation made at the outset of the proof, we conclude the argument.
\end{proof}

\begin{lem}\label{third interpolation lemma} Let $(\mathcal{M},\tau)$ be a finite von Neumann algebra and let $(\mathcal{N},\nu)$ be a semifinite atomless one. If $V:L_q(\mathcal{M})\to L_q(\mathcal{N})$ and $V:L_p(\mathcal{M})\to L_p(\mathcal{N})$ are bounded linear maps, then for every $E\in {\rm Int}(L_p,L_q),$ $1\leq p,q\leq\infty,$
$$\|V(x)\|_{Z_E^\infty(\mathcal{N})}\lesssim_E \|x\|_{E(\mathcal{M})},\quad x\in E(\mathcal{M}).$$
\end{lem}
\begin{proof} Fix an element $x=x^*\in E(\mathcal{M})$ and define the projections ${\bf p}$ and ${\bf q}$ as in the proof of Lemma \ref{second interpolation lemma}. Consider the operator
$$W:z\to {\bf q}\cdot Vz\cdot {\bf q},\quad z\in (L_p+L_q)(\mathcal{M}).$$
By the assumption, $W:L_q(\mathcal{M})\to L_q({\bf q}\mathcal{N}{\bf q})$ and $W:L_p(\mathcal{M})\to L_p({\bf q}\mathcal{N}{\bf q})$ are bounded. Since $E$ is an interpolation space between $L_q(0,1)$ and $L_p(0,1),$ it follows from \cite[Theorem 3.2]{DDP-Int} that $W:E(\mathcal{M})\to E({\bf q}\mathcal{N}{\bf q})$ is a bounded map.

Note that $\mu({\bf q}\cdot Vx\cdot {\bf q})$ lives on the interval $(0,1).$ For definiteness, let $p\leq q$ so that $E\subset L_p.$ We have
$$\|(1-{\bf q})\cdot Vx\|_{L_{\infty}(\mathcal{N})}\leq\|(1-{\bf p})\cdot Vx\|_{L_{\infty}(\mathcal{N})}\leq$$
$$\leq\mu(1,Vx)\leq\|Vx\|_p\lesssim_E\|x\|_p\lesssim\|x\|_{E(\mathcal{M})}.$$
We arrive at
\begin{eqnarray*}
\|Vx\|_{Z_E^{\infty}(\mathcal{N})}&\leq&\|{\bf q}\cdot Vx\cdot {\bf q}\|_{Z_E^{\infty}(\mathcal{N})}+\|{\bf q}\cdot Vx\cdot (1-{\bf q})\|_{Z_E^{\infty}(\mathcal{N})}+\|(1-{\bf q})\cdot Vx\|_{Z_E^{\infty}(\mathcal{N})}\\
&\leq&\|Wx\|_{E({\bf q}\mathcal{N}{\bf q})}+2\|(1-{\bf q})\cdot Vx\|_{L_{\infty}(\mathcal{N})}\\
&\leq& \|Wx\|_{E({\bf q}\mathcal{N}{\bf q})}+\|x\|_{E(\mathcal{M})}\lesssim_E \|x\|_{E(\mathcal{M})}.
\end{eqnarray*}
This proves the assertion for every self-adjoint $x\in E(\M).$ Arguing as at the end of Lemma \ref{first interpolation lemma}, we conclude the proof for an arbitrary $x\in E(\M).$
\end{proof}

\begin{lem}\label{fourth interpolation lemma} Let $(\mathcal{M},\tau)$ be a finite von Neumann algebra and let $(\mathcal{N},\nu)$ be a semifinite atomless one. If $V:L_p(\mathcal{N})\to L_p(\mathcal{M})$ and $V:L_q(\mathcal{N})\to L_q(\mathcal{M})$ are bounded linear maps, then for every $E\in {\rm Int}(L_p,L_q),$ $1\leq p,q\leq\infty,$
$$\|V(x)\|_{E(\mathcal{M})}\lesssim_E \|x\|_{Z_E^1(\mathcal{N})},\quad x\in Z_E^1(\mathcal{N}).$$
\end{lem}
\begin{proof} Fix a self-adjoint element $x=x^*\in Z_E^1(\mathcal{N})$ and choose a projection ${\bf q}$ such that $x{\bf q}={\bf q}x,$ $\nu({\bf q})=1$ and $|x|{\bf q}\geq\mu(1,x){\bf q}.$ Similarly, the operator $V:E({\bf q}\mathcal{N}{\bf q})\to E(\mathcal{M})$ is bounded and, therefore,
$$\|V({\bf q}x{\bf q})\|_{E(\M)}\lesssim_E\|\mu({\bf q}x{\bf q})\|_E=\|\mu(x)\chi_{(0,1)}\|_E.$$
For definiteness, we assume that $p<q$ so that $L_q\subset E.$ We have
\begin{eqnarray*}
\|V((1-{\bf q})x(1-{\bf q}))\|_{E(\M)}&\lesssim_E&\|V((1-{\bf q})x(1-{\bf q}))\|_{L_q(\M)}
\\&\lesssim&\|(1-{\bf q})x(1-{\bf q})\|_{L_q(\N)}
=\|\mu(x)\chi_{(1,\infty)}\|_{L_q}
\\&\leq&\|x\|_{L_1(\N)}.
\end{eqnarray*}
Hence,
$$\|Vx\|_E=\|V({\bf q}x{\bf q})+V((1-{\bf q})x(1-{\bf q}))\|_E\leq\|\mu(x)\chi_{(0,1)}\|_E+\|x\|_{L_1(\N)}\lesssim\|x\|_{Z_E^1(\N)}.$$
This proof is complete.
\end{proof}

\section{Johnson-Schechtman inequality for noncommutative symmetric spaces}\label{js section}

This section is devoted to prove the noncommutative Johnson-Schechtman inequality, Theorem \ref{js thm}. We first give two lemmas.

\begin{lem}\label{reduction to kh} Let $(\M,\tau)$ be a noncommutative probability space. If $\{x_k\}_{k\geq0}\subset L_q(\M)$ are independent mean zero random variables, then for $1\leq q\leq\infty,$
\begin{equation}\label{rtkheq}
\|\sum_{k\geq0}x_k\|_{L_q(\M)}\leq 2\|\sum_{k\geq0}x_k\otimes r_k\|_{L_q(\M\bar{\otimes}L_\infty(0,1))},
\end{equation}
where $\{r_k\}_{k\geq0}$ is the Rademacher sequence on $(0,1)$.
\end{lem}
\begin{proof} By \cite[Lemma 1.2]{JX}, for $1\leq q\leq \infty,$ we have
$$\Big\|\sum_{k=m}^nx_k\Big\|_{L_q(\M)}\leq 2\Big\|\sum_{k=m}^n\epsilon_k x_k\Big\|_{L_q(\M)}$$
for every choice of $\epsilon_k\in\{-1,1\}.$ Let $\Upsilon$ be the set of all sequences $\{\epsilon_k\}_{k=m}^n\subset\{-1,1\}.$ Observing that the cardinality of $\Upsilon$ is $2^{n-m+1},$ we obtain
\begin{equation*}
\Big\|\sum_{k=m}^nx_k\Big\|_{L_q(\M)}^q\leq 2^q\cdot 2^{m-n-1}\sum_{\epsilon\in\Upsilon}\Big\|\sum_{k=m}^n\epsilon_k x_k\Big\|_{L_q(\M)}^q=2^q\Big\|\sum_{k=m}^nx_k\otimes r_k\Big\|_{L_q(\M\bar{\otimes}L_\infty(0,1))}^q.
\end{equation*}

If the series on the right hand side of \eqref{rtkheq} converges in $L_q(\M\bar{\otimes}L_\infty(0,1)),$ then the partial sums of the series on the left hand side of \eqref{rtkheq} form a Cauchy sequence in $L_q(\M).$ Thus, the series on the left hand side of \eqref{rtkheq} also converges in $L_q(\M)$. Taking the limit in the last inequality as $n\to\infty$ and putting $m=0,$ we conclude the proof.
\end{proof}

The following lemma plays a crucial role in the proof of main results.

\begin{lem}\label{2^N lemma} Let $(\M,\tau)$ be a noncommutative probability space. If $\{x_k\}_{k\geq0}\subset L_{2^N}(\M)$ are independent mean zero random variables, then
$$\big\|\sum_{k\geq0} x_k\big\|_{L_{2^N}(\M)}\leq c_N\Big\|\sum_{k\geq0}x_k\otimes e_k\Big\|_{(L_{2^N}\cap L_2)(\M\bar{\otimes}\ell_\infty)},\quad N\in \mathbb{N},$$
where $c_N$ denotes a constant dependent only on $N.$
\end{lem}
\begin{proof} It follows from Definition \ref{NC defi} that independent mean zero random variables are pairwise orthogonal with respect to the usual inner product in $L_2(\mathcal{M}).$ Thus, for $N=1,$ the assertion obviously holds (with $c_1=1$). We prove the assertion by induction with respect to $N\in\mathbb{N}.$ By Lemma \ref{reduction to kh} and Khinchine inequality due to Lust-Piquard \cite{LP}, we have
$$\big\|\sum_{k\geq0}x_k\big\|_{L_{2^N}(\M)}\leq c_N'\Big\|\Big(\sum_{k\geq0}x_k^2\Big)^{\frac12}\Big\|_{L_{2^N}(\M)}=c_N'\Big\|\sum_{k\geq0}x_k^2\Big\|_{L_{2^{N-1}}(\M)}^{\frac12}.$$
Define mean zero random variables $y_k=x_k^2-\tau(x_k^2)\in L_{2^{N-1}}(\M).$ It follows directly from Definition \ref{NC defi} that so-constructed random variables are independent. Triangle inequality in the space $L_{2^{N-1}}(\M)$ yields
$$\big\|\sum_{k\geq0}x_k^2\big\|_{L_{2^{N-1}}(\M)}\leq\big\|\sum_{k\geq0}y_k\big\|_{L_{2^{N-1}}(\M)}+\big\|\sum_{k\geq0}x_k\otimes e_k\big\|_{L_2(\M\bar{\otimes}\ell_\infty)}^2.$$
Applying induction for the sequence $\{y_k\}_{k\geq0},$ we obtain that
$$\big\|\sum_{k\geq0}x_k^2\big\|_{L_{2^{N-1}}(\M)}\leq c_{N-1}\big\|\sum_{k\geq0}y_k\otimes e_k\big\|_{(L_{2^{N-1}}\cap L_2)(\M\bar{\otimes}\ell_\infty)}+\big\|\sum_{k\geq0}x_k\otimes e_k\big\|_{L_2(\M\bar{\otimes}\ell_\infty)}^2.$$
Now we estimate the first term on the right hand side of the above inequality. By the triangle inequality,
$$\big\|\sum_{k\geq0}y_k\otimes e_k\big\|_{L_{2^{N-1}}\cap L_2}\leq\big\|\sum_{k\geq0}x_k^2\otimes e_k\big\|_{L_{2^{N-1}}\cap L_2}+\big\|\sum_{k\geq0}\tau(x_k^2)\otimes e_k\big\|_{L_{2^{N-1}}\cap L_2}.$$
Noting that
$$\sum_{k\geq0}\tau(x_k^2)\otimes e_k\prec\prec \sum_{k\geq0}x_k^2\otimes e_k,$$
we have that
$$\big\|\sum_{k\geq0}\tau(x_k^2)\otimes e_k\big\|_{(L_{2^{N-1}}\cap L_2)(\M\bar{\otimes}\ell_\infty)}\leq \big\|\sum_{k\geq0}x_k^2\otimes e_k\big\|_{(L_{2^{N-1}}\cap L_2)(\M\bar{\otimes}\ell_\infty)}.$$
Consequently,
\begin{eqnarray*}
\big\|\sum_{k\geq0}y_k\otimes e_k\big\|_{(L_{2^{N-1}}\cap L_2)(\M\bar{\otimes}\ell_\infty)}&\leq& 2\big\|\sum_{k\geq0}x_k^2\otimes e_k\big\|_{(L_{2^{N-1}}\cap L_2)(\M\bar{\otimes}\ell_\infty)}\\
&=&2\big\|\sum_{k\geq0}x_k\otimes e_k\big\|_{(L_{2^N}\cap L_4)(\M\bar{\otimes}\ell_\infty)}^2\\
&\leq&2\big\|\sum_{k\geq0}x_k\otimes e_k\big\|_{(L_{2^N}\cap L_2)(\M\bar{\otimes}\ell_\infty)}^2,
\end{eqnarray*}
where the last inequality follows from the obvious inequality
$$\|Z\|_4\leq\max\{\|Z\|_2,\|Z\|_{2^N}\}=\|Z\|_{L_2\cap L_{2^N}}.$$
Thus,
$$\big\|\sum_{k\geq0}x_k^2\big\|_{L_{2^{N-1}}(\M)}\leq (2c_{N-1}+1)\big\|\sum_{k\geq0}x_k\otimes e_k\big\|_{(L_{2^N}\cap L_2)(\M\bar{\otimes}\ell_\infty)}^2.$$
and, hence,
$$\big\|\sum_{k\geq0}x_k\big\|_{L_{2^N}(\M)}\leq c_N'(2c_{N-1}+1)^{\frac12}\big\|\sum_{k\geq0}x_k\otimes e_k\big\|_{(L_{2^N}\cap L_2)(\M\bar{\otimes}\ell_\infty)}.$$
Setting $c_N=c_N'(2c_{N-1}+1)^{\frac12},$ we complete the induction process and conclude the proof.
\end{proof}
We are now in a position to prove Theorem \ref{js thm}.

\begin{proof}[of Theorem \ref{js thm}] {\bf Step 1:} Let $\mathcal{M}_k$ be the unital von Neumann subalgebras generated by mean zero independent random variables $x_k$ and $\mathcal{E}_k$ be the conditional expectation onto $\mathcal M_k.$ Since the random variables $x_k,$ $k\geq0,$ are independent, then so are the subalgebras $\mathcal{M}_k,$ $k\geq0.$ Consider the operator $T:(L_1+L_2)(\mathcal{M}\bar{\otimes}\ell_\infty)\to L_1(\mathcal{M})$ defined by the formula\footnote{We prove the convergence of the series in $L_1(\mathcal{M})$ in the next paragraph.}
\begin{equation}\label{t def}
T\big(\sum_{k\geq0}z_k\otimes e_k\big)\to \sum_{k\geq0}\mathcal{E}_k(z_k)-\tau(z_k).
\end{equation}
We claim that
\begin{equation}\label{t map1}
\|Tz\|_{L_1(\M)}\leq 2\|z\|_{(L_1+L_2)(\mathcal{M}\bar{\otimes}\ell_\infty)},\quad z\in (L_1+L_2)(\mathcal{M}\bar{\otimes}\ell_\infty)
\end{equation}
and
\begin{equation}\label{t map2}
\|Tz\|_{L_{2^N}(\M)}\leq 2c_N\|z\|_{(L_{2^N}\cap L_2)(\mathcal{M}\bar{\otimes}\ell_\infty)},\quad z\in (L_{2^N}\cap L_2)(\mathcal{M}\bar{\otimes}\ell_\infty).
\end{equation}

First, we prove that the operator $T$ is well defined and also establish \eqref{t map1}. Fix $z\in (L_1+L_2)(\mathcal{M}\bar{\otimes}\ell_\infty)$ and $\varepsilon>0.$ Choose $a\in L_1(\mathcal{M}\bar{\otimes}\ell_\infty)$ and $b\in L_2(\mathcal{M}\bar{\otimes}\ell_\infty)$ such that $z=a+b$ and such that $$\|a\|_{L_1(\mathcal{M}\bar{\otimes}\ell_\infty)}+\|b\|_{L_2(\mathcal{M}\bar{\otimes}\ell_\infty)}
\leq\|z\|_{(L_1+L_2)(\mathcal{M}\bar{\otimes}\ell_\infty)}+\varepsilon.$$
It is immediate that
$$\big\|\sum_{k\geq0}\mathcal{E}_k(a_k)-\tau(a_k)\big\|_{L_1(\M)}\leq 2\sum_{k\geq0}\|a_k\|_{L_1(\M)}=2\|a\|_{L_1(\mathcal{M}\bar{\otimes}\ell_\infty)}.$$
Writing $b=\sum_{k\geq0}b_k\otimes e_k$ and noting that $\mathcal{E}_k(b_k)-\tau(b_k),$ $k\geq0,$ are mean zero independent random variables, we infer that
\begin{eqnarray*}
\big\|\sum_{k\geq0}\mathcal{E}_k(b_k)-\tau(b_k)\big\|_{L_1(\M)}&\leq&\big\|\sum_{k\geq0}\mathcal{E}_k(b_k)-\tau(b_k)\big\|_{L_2(\M)}\\
&=&\Big(\sum_{k\geq0}\big\|\mathcal{E}_k(b_k)-\tau(b_k)\big\|^2_{L_2(\M)}\Big)^{1/2}\\
&=&\Big\|\sum_{k\geq0}\big(\mathcal{E}_k(b_k)-\tau(b_k)\big)\otimes e_k\Big\|_{L_2(\mathcal{M}\bar{\otimes}\ell_\infty)}\\
&\leq&2\big\|\sum_{k\geq0} b_k\otimes e_k\big\|_{L_2(\mathcal{M}\bar{\otimes}\ell_\infty)},
\end{eqnarray*}
where the series on the left hand sides converges in $L_1(\mathcal{M}).$ It follows from the two preceding inequalities that
$$\big\|\sum_{k\geq0}\mathcal{E}_k(z_k)-\tau(z_k)\big\|_{L_1(\M)}\leq 2\big\|a\big\|_{L_1(\mathcal{M}\bar{\otimes}\ell_\infty)}+2\big\|b\big\|_{L_2(\mathcal{M}\bar{\otimes}\ell_\infty)}$$
and so the series on the left hand sides converges in $L_1(\mathcal{M}).$ Thus,
$$\big\|T\big(\sum_{k\geq0} z_k\otimes e_k\big)\big\|_{L_1(\mathcal{M})}\leq 2\big\|z\big\|_{(L_1+L_2)(\mathcal{M}\bar{\otimes}\ell_\infty)}+2\varepsilon.$$
Since $\varepsilon>0$ is arbitrarily small, \eqref{t map1} follows.

We now establish \eqref{t map2}. By Lemma \ref{2^N lemma}, we have
$$\|T(z)\|_{L_{2^N}(\M)}\leq c_N\Big\|\sum_{k\geq0}\big(\mathcal{E}_k(z_k)-\tau(z_k)\big)\otimes e_k\Big\|_{(L_{2^N}\cap L_2)(\mathcal{M}\bar{\otimes}\ell_\infty)}.$$
Note that every conditional expectation operator is a contraction on $L_1(\M)$ and on $L_\infty(\M).$ It is immediate that
$$\sum_{k\geq0}(\mathcal{E}_k(z_k)-\tau(z_k))\otimes e_k\prec\prec 2\sum_{k\geq0}z_k\otimes e_k.$$
Since $L_{2^N}\cap L_2$ is fully symmetric, it follows that
$$\|T(z)\|_{L_{2^N}(\M)}\leq 2c_N\big\|\sum_{k\geq0} z_k\otimes e_k\big\|_{(L_{2^N}\cap L_2)(\mathcal{M}\bar{\otimes}\ell_\infty)}.$$
This proves \eqref{t map2}.

{\bf Step 2:} By Step 1, the operator $T$ defined by the formula \eqref{t def} boundedly acts from $(L_1+L_2)(\mathcal{M}\bar{\otimes}\ell_\infty)$ into $L_1(\M)$ and from $(L_{2^N}\cap L_2)(\mathcal{M}\bar{\otimes}\ell_\infty)$ into $L_{2^N}(\M).$ Note that (see \eqref{intersect banach dual} and \eqref{sum banach dual})
$$(L_1+L_2)(\mathcal{M}\bar{\otimes}\ell_\infty)^*=(L_{\infty}\cap L_2)(\mathcal{M}\bar{\otimes}\ell_\infty),$$
and
$$(L_{2^N}\cap L_2)(\mathcal{M}\bar{\otimes}\ell_\infty)^*=(L_{\frac{2^N}{2^N-1}}+L_2)(\mathcal{M}\bar{\otimes}\ell_\infty).$$
Since $T:X\to Y$ implies $T^*:Y^*\to X^*$ and $\|T^*\|=\|T\|$,  by taking the adjoints, we have
\begin{equation}\label{tstar map1}
T^*:L_{\infty}(\mathcal{M})\to (L_{\infty}\cap L_2)(\mathcal{M}\bar{\otimes}\ell_\infty),
\end{equation}
and
\begin{equation}\label{tstar map2}
T^*:L_{\frac{2^N}{2^N-1}}(\mathcal{M})\to (L_{\frac{2^N}{2^N-1}}+L_2)(\mathcal{M}\bar{\otimes}\ell_\infty).
\end{equation}
are both bounded. There is no ambiguity in denoting these operators with the same letter because that they both have the form
\begin{equation}\label{tstar def}
T^*:w\to\sum_{k\geq0}\big(\mathcal{E}_k(w)-\tau(w)\big)\otimes e_k,
\end{equation}
where $w\in L_{\infty}(\mathcal{M})$ or, respectively, $w\in L_{\frac{2^N}{2^N-1}}(\mathcal{M}).$ To see the latter, for $z\in (L_1+L_2)(\mathcal{M}\bar{\otimes}\ell_\infty)$ and $w\in L_{\infty}(\mathcal{M})$ or $z\in (L_{2^N}\cap L_2)(\mathcal{M}\bar{\otimes}\ell_\infty)$ and $w\in L_{\frac{2^N}{2^N-1}}(\mathcal{M}),$ we have
$$\tau\big(\mathcal E_k(z_k)w\big)=\tau\big(\mathcal E_k(z_k)\mathcal E_k(w)\big)=\tau\big(z_k\mathcal{E}_k(w)\big)$$
and, therefore,
\begin{eqnarray*}
\tau\big(T(z)w\big)&=&\sum_{k\geq0}\tau\big(z_k\mathcal{E}_k(w)\big)-\tau(z_k)\tau(w)
=\sum_{k\geq0}\tau\Big(z_k\big(\mathcal{E}_k(w)-\tau(w)\big)\Big)
\\&=&(\tau\otimes\Sigma)\Big(\big(\sum_{k\geq0}z_k\otimes e_k\big)\cdot\Big(\sum_{k\geq0}\big(\mathcal{E}_k(w)-\tau(w)\big)\otimes e_k\Big)\Big).
\end{eqnarray*}
This shows \eqref{tstar def}.

{\bf Step 3:} Let $E\in{\rm Int}(L_1,L_q),$ $1\leq q<\infty.$ It follows\footnote{Let $2^N>q.$ If $W:L_1\to L_1$ and $W:L_{2^N}\to L_{2^N}$ are bounded linear operators, then $W:L_q\to L_q$ is also bounded because $L_q\in{\rm Int}(L_1,L_{2^N}).$ Since $E\in{\rm Int}(L_1,L_q),$ it follows that $W:E\to E$ is also bounded.} that $E\in {\rm Int}(L_1,L_{2^N})$ for some $N\in\mathbb{Z}^+.$ By Step 1 (see \eqref{t map1} and \eqref{t map2}) and Lemma \ref{first interpolation lemma}, we have that $T:Z_E^2(\mathcal{M}\bar{\otimes}\ell_\infty)\to E(\mathcal{M}).$ 

Recall that $x_k\in E(\M_k),$ $k\geq0,$ are independent mean zero random variables. It follows that $\mathcal{E}_k(x_k)-\tau(x_k)=x_k,$ $k\geq0.$ Thus,
$$T(\sum_{k\geq0}x_k\otimes e_k)=\sum_{k\geq0}x_k.$$
Applying the operator $T$ to the element $\sum_{k\geq0}x_k\otimes e_k,$ we conclude the proof of \eqref{maina}.
Let $E\in{\rm Int}(L_p,L_{\infty}),$ $1<p\leq\infty.$ It follows that $E\in{\rm Int}(L_{\frac{2^N}{2^N-1}},L_{\infty})$ for some $N\in\mathbb{N}.$ By Step 2 (see \eqref{tstar map1} and \eqref{tstar map2}) and Lemma \ref{second interpolation lemma}, we have that $T^*:E(\mathcal{M})\to Z_E^2(\mathcal{M}\bar{\otimes}\ell_\infty).$ We have
$$T^*(\sum_{k\geq0}x_k)=\sum_{k\geq0}x_k\otimes e_k.$$
Applying the operator $T^*$ to the element $\sum_{k\geq0}x_k,$ we conclude the proof of \eqref{mainb}. The claim in \eqref{mainc} is a trivial corollary of those in \eqref{maina} and \eqref{mainb}.

The proof of Theorem \ref{js thm} is complete.
\end{proof}

Applying Theorem \ref{js thm}, we easily establish Corollary \ref{js positive}.

\begin{proof}[of the Corollary \ref{js positive}] Clearly, the sequence $\{x_k-\tau(x_k)\}_{k\geq0}$ satisfies the assumptions of Theorem \ref{js thm}. Let $E\in{\rm Int}(L_1,L_q),$ $1\leq q<\infty.$ For simplicity of notations, let us assume that $\|1\|_E=1.$ It is clear from the triangle inequality in the space $E(\M)$ that
$$\|\sum_{k\geq0}x_k\|_{E(\M)}\leq\|\sum_{k\geq0}x_k-\tau(x_k)\|_{E(\M)}+\sum_{k\geq0}\tau(x_k).$$
Applying Theorem \ref{js thm} to the first term on the right hand side, we obtain that
$$\|\sum_{k\geq0}x_k\|_{E(\M)}\lesssim_E\Big\|\sum_{k\geq0}\big(x_k-\tau(x_k)\big)\otimes e_k\Big\|_{Z_E^2(\mathcal{M}\bar{\otimes}\ell_\infty)}+\big\|\sum_{k\geq0}x_k\otimes e_k\big\|_{L_1(\mathcal{M}\bar{\otimes}\ell_\infty)}.$$
Using triangle inequality in the space $Z_E^2(\mathcal{M}\bar{\otimes}\ell_\infty),$ we obtain that
$$\Big\|\sum_{k\geq0}\big(x_k-\tau(x_k)\big)\otimes e_k\Big\|_{Z_E^2(\mathcal{M}\bar{\otimes}\ell_\infty)}\leq 2\|\sum_{k\geq0}x_k\otimes e_k\|_{Z_E^2(\mathcal{M}\bar{\otimes}\ell_\infty)}\lesssim\|\sum_{k\geq0}x_k\otimes e_k\|_{Z_E^1(\mathcal{M}\bar{\otimes}\ell_\infty)}.$$
Consequently,
\begin{eqnarray*}
\|\sum_{k\geq0}x_k\|_{E(\M)}&\lesssim_E&\|\sum_{k\geq0}x_k\otimes e_k\|_{Z_E^1(\mathcal{M}\bar{\otimes}\ell_\infty)}+\big\|\sum_{k\geq0}x_k\otimes e_k\big\|_{L_1(\mathcal{M}\bar{\otimes}\ell_\infty)}\\
&\leq&2\|\sum_{k\geq0}x_k\otimes e_k\|_{Z_E^1(\mathcal{M}\bar{\otimes}\ell_\infty)}.
\end{eqnarray*}
This proves \eqref{cora}.

To see \eqref{corb}, suppose first that the series we consider contain only finitely many non-zero summands. Let $X=\sum_{k\geq0}x_k\otimes e_k.$ Set $x_{1k}=x_ke_{(\mu(1,X),\infty)}(x_k).$ For every $k\geq0,$ we have $x_k\leq x_{1k}+\mu(1,X)$ and, therefore,
$$\mu(X)\leq\mu(\sum_{k\geq0}x_{1k}\otimes e_k+\sum_{k\geq0}\mu(1,X)\otimes e_k)=\mu\Big(\sum_{k\geq0}x_{1k}\otimes e_k\Big)+\mu(1,X).$$
Thus,
$$\mu(X)\chi_{(0,1)}\leq \mu\Big(\sum_{k\geq0}x_{1k}\otimes e_k\Big)+\mu(1,X)\chi_{(0,1)}.$$
It follows from \cite[Lemma 36]{SZfreeK} that
$$\sum_{k\geq0}x_{1k}\otimes e_k\lhd \sum_{k\geq0}x_{1k}.$$
By \cite[Corollary 3.4.3]{LSZ}, we have
$$\big\|\sum_{k\geq0}x_{1k}\otimes e_k\big\|_{E(\mathcal{M}\bar{\otimes}\ell_\infty)}\leq \big\|\sum_{k\geq0}x_{1k}\big\|_{E(\M)}.$$
Therefore,
$$\|\mu(X)\chi_{(0,1)}\|_E\leq\|\sum_{k\geq0}x_{1k}\|_{E(\M)}+\mu(1,X)\leq\|\sum_{k\geq0}x_{1k}\|_{E(\M)}+\|X\|_1.$$
Given the obvious inequalities
$$\|\sum_{k\geq0}x_{1k}\|_{E(\M)}\leq \|\sum_{k\geq0}x_{k}\|_{E(\M)}\quad {\rm and}\quad\|X\|_1=\|\sum_{k\geq0}x_k\|_{L_1(\M)}\leq\|\sum_{k\geq0}x_k\|_{E(\M)},$$
we have
$$\|\sum_{k\geq0}x_k\otimes e_k\|_{Z_E^1(\mathcal{M}\bar{\otimes}\ell_\infty)}\leq \|\mu(X)\chi_{(0,1)}\|_E+ \|X\|_1\leq3\|\sum_{k\geq0}x_k\|_{E(\M)}.$$
This proves \eqref{corb} for the case of finitely many summands.

Let us now have infinitely many summands in our series. If $n,m\in\mathbb{Z}_+,$ then it follows from the above paragraph that
$$\|\sum_{k=n}^mx_k\otimes e_k\|_{Z_E^1(\mathcal{M}\bar{\otimes}\ell_\infty)}\leq 3\|\sum_{k=n}^mx_k\|_{E(\M)}.$$
Since the series $\sum_{k\geq0}x_k$ converges in $E(\M),$ it follows that its partial sums form a Cauchy sequence in $E(\M).$ Thus, partial sums of the series $\sum_{k\geq0}x_k\otimes e_k$ form a Cauchy sequence in $Z_E^1(\mathcal{M}\bar{\otimes}\ell_\infty).$ Passing $m\to\infty$ and letting $n=0,$ we conclude the proof of \eqref{corb}.

\eqref{corc} immediately follows from the combination of \eqref{cora} and \eqref{corb}. The proof of Corollary \ref{js positive} is complete.
\end{proof}

\begin{rmk}
It is surprising that the proof of Corollary \ref{js positive} (\ref{corb}) does not use the independence assumption. If $\M$ is a commutative von Neumann algebra, then Corollary \ref{js positive} (\ref{corb}) significantly extends the left hand side inequality in \cite[Theorem 1, (4)]{JS} (for the Banach space case).
\end{rmk}

\section{Khinchine inequality for symmetric spaces}\label{kh section}

In this section, we provide the proof of the noncommutative Khinchine inequality for symmetric spaces.

\begin{lem}\label{l bound} Define the operator $L:L_q(\mathcal{M}\bar{\otimes}\mathcal{L}(\ell_2))\to L_q(\mathcal{M}\bar{\otimes}\ell_{\infty}),$ $q\geq 2,$ by the formula
$$Lz=\sum_{k\geq0}z_{k,0}\otimes e_k,\quad z\in L_q(\mathcal{M}\bar{\otimes}\mathcal{L}(\ell_2)).$$
We have $\|L\|_{L_q\to L_q}\leq 1,$ $q\geq 2.$
\end{lem}
\begin{proof} We claim that $|Lz|^2\prec\prec |z|^2,$ $z\in L_q(\mathcal{M}\bar{\otimes}\mathcal{L}(\ell_2)).$ To see the claim, note that
$$|z\cdot(1\otimes e_{0,0})|^2=\big(\sum_{k\geq0}|z_{k,0}|^2\big)\otimes e_{00},\quad |Lz|^2=\sum_{k\geq0}|z_{k,0}|^2\otimes e_k.$$
It follows from \cite[Lemma 3.3.7]{LSZ} that
$$|Lz|^2=\sum_{k\geq0}|z_{k,0}|^2\otimes e_k\prec\prec\sum_{k\geq0}|z_{k,0}|^2=|z\cdot(1\otimes e_{0,0})|^2\prec\prec|z|^2.$$
This proves the claim.

Since $q\geq2,$ it follows that
$$\|Lz\|_{L_q(\mathcal{M}\bar{\otimes}\ell_{\infty})}\leq \|z\cdot(1\otimes e_{0,0})\|_{L_q(\mathcal{M}\bar{\otimes}\mathcal{L}(\ell_2))}\leq \|z\|_{L_q(\mathcal{M}\bar{\otimes}\mathcal{L}(\ell_2))}.$$
Thus, $L$ is bounded.
\end{proof}

\begin{proof}[of Theorem \ref{khinchine}] Let $\mathcal{M}_k$ be the unital von Neumann subalgebras generated by independent mean zero random random variables $x_k$ and let $\mathcal{E}_k$ be the conditional expectation onto $\mathcal M_k.$ Since the random variables $x_k,$ $k\geq0,$ are independent, then so are the subalgebras $\mathcal{M}_k,$ $k\geq0.$

By assumption, our space $E\in{\rm Int}(L_p,L_q),$ $1<p\leq q<\infty.$ Without loss of generality, $1<p\leq 2,$ $2\leq q<\infty$ and $\frac1p+\frac1q=1.$ That is, $p$ and $q$ are fixed during the proof.

{\bf Step 1:} Let $\{e_{ij}\}_{i,j\geq0}$ denote the standard basis of $\mathcal{L}(\ell_2).$ Define the operator $S:L_q(\mathcal{M})\to L_q(\mathcal{M}\bar{\otimes}\mathcal{L}(\ell_2))$ by the formula
\begin{equation}\label{s def}
S:z\to\sum_{k\geq0}\big(\mathcal{E}_k(z)-\tau(z)\big)\otimes e_{k,0},\quad z\in L_q(\mathcal{M}).
\end{equation}
We claim that $S:L_q(\mathcal{M})\to L_q(\mathcal{M}\bar{\otimes}\mathcal{L}(\ell_2))$ is bounded. To this end, note that
$$|Sz|=\Big(\sum_{k\geq0}|\mathcal{E}_k(z)-\tau(z)|^2\Big)^{1/2}\otimes e_{0,0},$$
and, hence,
$$\|Sz\|_{L_q(\mathcal{M}\bar{\otimes}\mathcal{L}(\ell_2))}=\Big\|\sum_{k\geq0}|\mathcal{E}_k(z)-\tau(z)|^2\Big\|^{1/2}_{L_\frac{q}{2}(\M)}.$$
Clearly, the random variables $\mathcal{E}_k(z)-\tau(z)\in L_q(\mathcal{M}),$ $k\geq0,$ are independent and mean zero. Since $2\leq q<\infty,$ it follows from the Corollary \ref{js positive} \eqref{cora} that
\begin{eqnarray*}
\|Sz\|_{L_q(\mathcal{M}\bar{\otimes}\mathcal{L}(\ell_2))}&\lesssim_q&\Big\|\sum_{k\geq0}|\mathcal{E}_k(z)-\tau(z)|^2\otimes e_k\Big\|^{1/2}_{Z_{L_{\frac{q}{2}}}^1(\mathcal{M}\bar{\otimes}\ell_\infty)}\\
&\approx&\Big\|\sum_{k\geq0}(\mathcal{E}_k(z)-\tau(z))\otimes e_k\Big\|_{Z_{L_q}^2(\mathcal{M}\bar{\otimes}\ell_\infty)}=\|T^*z\|_{Z_{L_q}^2(\mathcal{M}\bar{\otimes}\ell_\infty)},
\end{eqnarray*}
where $T^*$ is the operator introduced in (Step 2 of) the proof of Theorem \ref{js thm}. Since $q>1,$ it was established there that $T^*:L_q(\mathcal{M})\to Z_{L_q}^2(\mathcal{M}\bar{\otimes}\ell_\infty)$ is bounded. Therefore, we have
$$\|Sz\|_{L_q(\mathcal{M}\bar{\otimes}\mathcal{L}(\ell_2))}\lesssim_q\|T^*\|_{L_q\to Z_{L_q}^2}\|x\|_{L_q(\mathcal{M})}.$$

{\bf Step 2:} Using Step 1 and taking adjoints, we infer that the operator $S^*:L_p(\mathcal{M}\bar{\otimes}\mathcal{L}(\ell_2))\to L_p(\mathcal{M})$ is bounded. Similarly to \eqref{tstar def}, a simple computation shows that
\begin{equation}\label{sstar def}
S^*z=\sum_{k\geq0}\mathcal{E}_k(z_{k,0})-\tau(z_{k,0}),\quad z\in L_p(\mathcal{M}\bar{\otimes}\mathcal{L}(\ell_2)).
\end{equation}
We claim that the operator $S^*:L_q(\mathcal{M}\bar{\otimes}\mathcal{L}(\ell_2))\to L_q(\mathcal{M})$ given by the formula \eqref{sstar def} is also bounded. Again taking the adjoints and arguing as in \eqref{tstar def}, we would infer that the operator $S:L_p(\mathcal{M})\to L_p(\mathcal{M}\bar{\otimes}\mathcal{L}(\ell_2)),$ given by the formula \eqref{s def} is also bounded.

Clearly, the random variables $\mathcal{E}_k(z_{k,0})-\tau(z_{k,0})\in L_q(\mathcal{M}),$ $k\geq0,$ are independent and mean zero. It follows from Theorem \ref{js thm} that
$$\|S^*z\|_{L_q(\mathcal{M})}\approx\|\sum_{k\geq0}(\mathcal{E}_k(z_{k,0})-\tau(z_{k,0}))\otimes e_k\|_{Z_{L_q}^2(\mathcal{M}\bar{\otimes}\ell_{\infty})}.$$
However,
$$\sum_{k\geq0}(\mathcal{E}_k(z_{k,0})-\tau(z_{k,0}))\otimes e_k\prec\prec 2\sum_{k\geq0}z_{k,0}\otimes e_k=2Lz=2L(z\cdot (1\otimes e_{0,0})).$$
Since $L:L_q(\mathcal{M}\bar{\otimes}\mathcal{L}(\ell_2))\to L_q(\mathcal{M}\bar{\otimes}\ell_{\infty})$ is bounded, it follows from Lemma \ref{l bound} that
$$\|S^*z\|_q\lesssim\max\{\|L(z\cdot(1\otimes e_{0,0}))\|_q,\|L(z\cdot(1\otimes e_{0,0}))\|_2\}\leq$$
$$\leq\max\{\|z\cdot(1\otimes e_{0,0})\|_q,\|z\cdot(1\otimes e_{0,0})\|_2\}=\|z\cdot(1\otimes e_{0,0})\|_q\leq\|z\|_q.$$

{\bf Step 3:} It is established in Step 1 that $S:L_q(\mathcal{M})\to L_q(\mathcal{M}\bar{\otimes}\mathcal{L}(\ell_2))$ is bounded; it is established in Step 2 that $S:L_p(\mathcal{M})\to L_p(\mathcal{M}\bar{\otimes}\mathcal{L}(\ell_2))$ is bounded. Using Lemma \ref{third interpolation lemma}, we infer that $S:E(\mathcal{M})\to Z_E^{\infty}(\mathcal{M}\bar{\otimes}\mathcal{L}(\ell_2))$ is bounded. Since our random variables $x_k,$ $k\geq0,$ generate the respective algebras $\M_k,$ $k\geq0,$ and are mean zero, it follows from the definition of the operator $S$ that
$$S(\sum_{k\geq0}x_k)=\sum_{k\geq0}x_k\otimes e_{k,0}.$$
Applying the operator $S$ to the element $\sum_{k\geq0}x_k,$ we arrive at
$$\Big\|\Big(\sum_{k\geq0}x_k^2\Big)^{\frac12}\Big\|_{E(\mathcal{M})}=\|\sum_{k\geq0}x_k\otimes e_{k,0}\|_{Z_E^{\infty}(\mathcal{M}\bar{\otimes}\mathcal{L}(\ell_2))}\lesssim_E\|\sum_{k\geq0}x_k\|_{E(\mathcal{M})}.$$

It is established in Step 2 that $S^*:L_p(\mathcal{M}\bar{\otimes}\mathcal{L}(\ell_2))\to L_p(\mathcal{M})$ and $S^*:L_q(\mathcal{M}\bar{\otimes}\mathcal{L}(\ell_2))\to L_q(\mathcal{M})$ are bounded. Using Lemma \ref{fourth interpolation lemma}, we infer that $S^*:Z_E^1(\mathcal{M}\bar{\otimes}\mathcal{L}(\ell_2))\to E(\mathcal{M})$ is also bounded. We have
$$S^*(\sum_{k\geq0}x_k\otimes e_{k,0})=\sum_{k\geq0}x_k.$$
Applying the operator $S^*$ to the element $\sum_{k\geq0}x_k\otimes e_{k,0},$ we arrive at
$$\|\sum_{k\geq0}x_k\|_{E(\mathcal{M})}\lesssim_E\|\sum_{k\geq0}x_k\otimes e_{k,0}\|_{Z_E^1(\mathcal{M}\bar{\otimes}\mathcal{L}(\ell_2))}=\Big\|\Big(\sum_{k\geq0}x_k^2\Big)^{\frac12}\Big\|_{E(\mathcal{M})}.$$
This proves the theorem.
\end{proof}

\section{Johnson-Schechtman and Khinchine inequalities for modulars}\label{modular section}

In this section, we prove the noncommutative Johnson-Schechtman and Khinchine inequalities associated with an Orlicz function $\Phi.$ We need interpolation lemmas for $\Phi$-moments similar to those given in Section \ref{interpol section}. The proof below have subtle differences with those in Section \ref{interpol section}.

Let $1\leq p\leq q<\infty.$ We introduce the Orlicz function
\begin{equation}\label{M}
M_{p,q}(t)=
\begin{cases}
pt^q,& t\in[0,1)\\
qt^p+p-q,&  t\in[1,\infty).
\end{cases}
\end{equation}
It is clear that $M_{p,q}$ is a $p$-convex and $q$-concave Orlicz function.

\begin{lem}\label{m simple est} Let $(\mathcal{M}_1,\tau_1)$ and $(\mathcal{M}_2,\tau_2)$ be semifinite von Neumann algebras. If $W:L_p(\mathcal{M}_1)\to L_p(\mathcal{M}_2)$ and $W:L_q(\mathcal{M}_1)\to L_q(\mathcal{M}_2)$ are linear contractions, then
$$\tau_2\Big(M_{p,q}(|Wx|)\Big)\lesssim_{p,q}\tau_1\Big(M_{p,q}(|x|)\Big),\quad x\in L_{M_{p,q}}(\M_1).$$
\end{lem}
\begin{proof} Split $x=x_1+x_2,$ where $x_1=xe_{[0,1]}(|x|)$ and $x_2=xe_{(1,\infty)}(|x|).$ By \cite[Theorem 3.3.3]{LSZ}, we have
$$|Wx|\prec\prec\mu(Wx_1)+\mu(Wx_2).$$
Let $M_0(t)=M_{p,q}(t^{1/q}),$ $t\geq0.$ By $q$-concavity, $M_0$ is a concave function. Using the basic properties of concave functions, it is clear that
$$M_{p,q}(t_1+t_2)\leq M_0\big(2^{q-1}(t_1^q+t_2^q)\big)\leq 2^{q-1}M_0(t_1^q+t_2^q)\leq 2^{q-1}\big(M_{p,q}(t_1)+M_{p,q}(t_2)\big)$$
for all $t_1,t_2>0.$ Therefore, by \eqref{majorization phi} we have
$$\tau_2\Big(M_{p,q}(|Wx|)\Big)\leq 2^{q-1}\Big(\tau_2\Big(M_{p,q}(|Wx_1|)\Big)+\tau_2\Big(M_{p,q}(|Wx_2|)\Big)\Big).$$
Since $M_{p,q}(t)\leq qt^p$ and $M_{p,q}(t)\leq qt^q$ for every $t>0,$ it follows that
$$\tau_2\Big(M_{p,q}(|Wx|)\Big)\leq q2^{q-1}\Big(\tau_2(|Wx_1|^q)+\tau_2(|Wx_2|^p)\Big).$$
Since $W:L_p(\mathcal{M}_1)\to L_p(\mathcal{M}_2)$ and $W:L_q(\mathcal{M}_1)\to L_q(\mathcal{M}_2)$ are contractions, it follows that
$$\tau_2\Big(M_{p,q}(|Wx|)\Big)\leq q2^{q-1}\Big(\tau_1(|x_1|^q)+\tau_1(|x_2|^p)\Big).$$

Since the spectrum of $|x_1|$ is in $[0,1]$ and since the spectrum of $x_2$ is in $[1,\infty),$ it follows that
$$M_{p,q}(|x_1|)=p|x_1|^q,\quad M_{p,q}(|x_2|)\geq p|x_2|^p.$$
Thus,
$$\tau_2(M_{p,q}(|Wx|))\leq\frac{q2^{q-1}}{p}\Big(\tau_1(M_{p,q}(|x_1|))+\tau_1(M_{p,q}(|x_2|)))\leq\frac{q2^q}{p}\tau_1(M_{p,q}(|x|)).$$
\end{proof}

The following lemma extends $\Phi$-moment interpolation theorem \cite[Theorem 2.1]{BC}; our methods are quite distinct from those used in \cite{BC}.

\begin{lem}\label{alpha interpolation lemma} Let $(\mathcal{M}_1,\tau_1)$ and $(\mathcal{M}_2,\tau_2)$ be semifinite von Neumann algebras and and let $\Phi$ be $p$-convex and $q$-concave Orlicz function, $1\leq p\leq q<\infty$. If $W:L_p(\mathcal{M}_1)\to L_p(\mathcal{M}_2)$ and $W:L_q(\mathcal{M}_1)\to L_q(\mathcal{M}_2),$ are bounded linear operators, then
$$\tau_2(\Phi(|Wx|))\lesssim_{\Phi}\tau_1(\Phi(|x|)),\quad x\in L_\Phi(\M_1).$$
\end{lem}
\begin{proof} Define the function $\phi:(0,\infty)\to(0,\infty)$ by setting $\Phi(t)=t^p\phi(t^{q-p}),$ $t>0.$ Since $\Phi$ is $p$-convex and $q$-concave, it follows that $\phi$ is quasi-concave (see \cite[Lemma 6]{AS2014}). By \cite[Theorem II.1.1]{KPS}, there exists an increasing concave function $\phi_0$ such that $\frac12\phi_0\leq\phi\leq\phi_0.$ For simplicity, suppose that $\phi_0(0)=0$ and $\phi_0(t)=o(t),$ $t\to\infty.$ By \cite[Lemma 5.4.3]{BergLofstrom}, we have
$$\phi_0(t)=\int_0^{\infty}\min\{t,\lambda\}d(-\phi_0'(\lambda)).$$
Note that $\min\{t^p,t^q\}\approx_{p,q} M_{p,q}(t)$ for all $t\geq0,$ we have
\begin{eqnarray*}
\Phi(t)&\approx& \int_0^{\infty}\min\{t^q,\lambda t^p\}d(-\phi_0'(\lambda))\stackrel{\lambda=s^{p-q}}{=}\int_0^{\infty}\min\{t^q,s^{p-q}t^p\}d(\phi_0'(s^{p-q}))
\\&=&\int_0^{\infty}\min\{(ts)^q,(ts)^p\}d\nu(s)\approx_{p,q}\int_0^{\infty}M_{p,q}(ts)d\nu(s),
\end{eqnarray*}
where the positive measure $\nu$ is defined by setting $d\nu(s)=s^{-q}d(\phi_0'(s^{p-q})).$ Hence,
$$\tau_1(\Phi(|x|))\approx\int_0^{\infty}\tau_1\Big(M_{p,q}(s|x|)\Big)d\nu(s).$$
Therefore, by Lemma \ref{m simple est}, we have
\begin{eqnarray*}
\tau_2(\Phi(|Wx|))\approx\int_0^{\infty}\tau_2\Big(M_{p,q}(s|Wx|)\Big)d\nu(s)
\leq\int_0^{\infty}\tau_1\Big(M_{p,q}(s|x|)\Big)d\nu(s)=\tau_1(\Phi(|x|)),
\end{eqnarray*} which is our desired inequality and the proof is complete.
\end{proof}

\begin{lem}\label{beta interpolation lemma} Let $(\mathcal{M}_1,\tau_1)$ and $(\mathcal{M}_2,\tau_2)$ be semifinite von Neumann algebras. Let $\Phi$ be a $p$-convex and $q$-concave Orlicz function, $1\leq p\leq q<\infty.$ If $W:L_{\infty}(\mathcal{M}_1)\to L_{\infty}(\mathcal{M}_2)$ and $W:L_p(\mathcal{M}_1)\to L_p(\mathcal{M}_2),$ are bounded linear maps, then
$$\tau_2(\Phi(|Wx|))\lesssim_{\Phi}\tau_1(\Phi(|x|)),\quad x\in L_\Phi(\M_1).$$
\end{lem}
\begin{proof}  Let us define Peetre $K-$functional by the usual formula.
$$K(t,z,L_p,L_{\infty})=\inf\{\|z_1\|_p+t\|z_2\|_{\infty}:\ z=z_1+z_2\}.$$
Without loss of generality, $W:L_{\infty}(\mathcal{M}_1)\to L_{\infty}(\mathcal{M}_2)$ and $W:L_p(\mathcal{M}_1)\to L_p(\mathcal{M}_2)$ are contractions. Therefore, we have
$$K(t,Wx,L_p(\mathcal{M}_2),L_{\infty}(\mathcal{M}_2))\leq K(t,x,L_p(\mathcal{M}_1),L_{\infty}(\mathcal{M}_1)),\quad t>0.$$
Taking into account that (see e.g. the computation leading to the formula (3.5) in II.3.1 in \cite{KPS})
$$K(t,z,L_p,L_{\infty})\approx_p\Big(\int_0^{t^p}\mu^p(s,z)ds\Big)^{\frac1p}, \quad z\in (L_p+L_\infty)(\M),$$
we infer that
$$|Wx|^p\prec\prec c_p|x|^p,\quad x\in (L_p+L_{\infty})(\M_1).$$
Let $\Phi_0(t)=\Phi(t^{\frac1p}),$ $t\geq0.$ It follows from the assumption of $p$-convexity of $\Phi$ that $\Phi_0$ is convex. Thus, by \eqref{majorization phi},
$$\tau_2(\Phi(|Wx|))=\tau_2(\Phi_0(|Wx|^p))\leq \tau_1(\Phi_0(c_p|x|^p))=\tau_1(\Phi(c_p^{\frac1p}|x|))\lesssim_\Phi\tau_1(\Phi(|x|)),$$
where the last inequality follows from the assumption of $q$-concavity of $\Phi.$
\end{proof}

\begin{lem}\label{fifth interpolation lemma} Let $(\mathcal{M},\tau)$ be a finite von Neumann algebra and let $(\mathcal{N},\nu)$ be a semifinite atomless one. Let $\Phi$ be a  $q$-concave Orlicz function with $2\leq q<\infty.$ If $V:(L_1+L_2)(\mathcal{N})\to L_1(\mathcal{M})$ and $V:(L_q\cap L_2)(\mathcal{N})\to L_q(\mathcal{M}),$ are bounded linear maps, then
$$\tau\big(\Phi(|Vx|)\big)\lesssim_{\Phi}\int_0^1\Phi(\mu(t,x))dt+\Phi(\|x\|_{L_1+L_2}).$$
\end{lem}
\begin{proof} Without loss of generality, let $x$ be self-adjoint and fix $x=x^*.$ Choose a projection ${\bf q}$ such that $x{\bf q}={\bf q}x,$ $\nu({\bf q})=1$ and $|x|{\bf q}\geq\mu(1,x){\bf q}.$ The operator $V:L_1({\bf q}\mathcal{N}{\bf q})\to L_1(\mathcal{M})$ is bounded and also $V:L_q({\bf q}\mathcal{N}{\bf q})\to L_q(\mathcal{M}).$ By Lemma \ref{alpha interpolation lemma}, we have
$$\tau\Big(\Phi(V({\bf q}x{\bf q}))\Big)\lesssim_{\Phi}\nu(\Phi({\bf q}x{\bf q}))=\int_0^1\Phi(\mu(t,x))dt.$$
Let $\Phi_0(t)=\Phi(t^{\frac1q})$, $t\geq0.$ It follows from the assumption of $q$-concavity of $\Phi$ that $\Phi_0$ is concave. For an arbitrary $z\in L_q(\M),$ by the Jensen inequality we have
$$\tau(\Phi(|z|))=\tau(\Phi_0(|z|^q))=\int_0^1\Phi_0(\mu(t,z^q))dt\leq\Phi_0\Big(\int_0^1\mu(t,z^q)dt\Big)= \Phi(\|z\|_q).$$
Therefore,
\begin{eqnarray*}
\tau\Big(\Phi\Big(V((1-{\bf q})x(1-{\bf q}))\Big)\Big)&\leq&\Phi\Big(\|V((1-{\bf q})x(1-{\bf q}))\|_q\Big)\\&\leq&\Phi\Big(\|(1-{\bf q})x(1-{\bf q})\|_{L_q\cap L_2}\Big)\\
&=&\Phi\Big(\|\mu(x)\chi_{(1,\infty)}\|_{L_q\cap L_2}\Big)\lesssim\Phi(\|x\|_{L_1+L_2}).
\end{eqnarray*}
Since the assumption of $q$-concavity of $\Phi$ implies $\Delta_2$-condition, it follows from \eqref{phi quasi-trangle} that
$$\tau\Big(\Phi(Vx)\Big)\lesssim_{\Phi}\tau\Big(\Phi(V({\bf q}x{\bf q}))\Big)+\tau\Big(\Phi(V((1-{\bf q})x(1-{\bf q})))\Big)$$
$$\lesssim_{\Phi}\int_0^1\Phi(\mu(t,x))dt+\Phi(\|x\|_{L_1+L_2}).$$
\end{proof}

\begin{lem}\label{sixth interpolation lemma} Let $(\mathcal{M},\tau)$ be a finite von Neumann algebra and let $(\mathcal{N},\nu)$ be a semifinite atomless one. Let $\Phi$ be a  $p$-convex and $q$-concave Orlicz function with $1<p\leq 2\leq q<\infty.$ If $V:L_p(\mathcal{M})\to(L_p+L_2)(\mathcal{N})$ and $V:L_{\infty}(\mathcal{M})\to(L_2\cap L_{\infty})(\mathcal{N}),$ are bounded linear maps, then
$$\int_0^1\Phi(\mu(t,Vx))dt+\Phi(\|Vx\|_{L_1+L_2})\lesssim_{\Phi}\tau\big(\Phi(|x|)\big),\quad x\in L_\Phi(\M).$$
\end{lem}
\begin{proof} Without loss of generality, $V$ maps self-adjoint operators to self-adjoint ones. Fix $x=x^*$ and consider the spectral projection of $|Vx|$ associated with the interval $(\mu(1,Vx),\infty),$
$${\bf p}:=e_{(\mu(1,Vx),\infty)}(|Vx|).$$
It is clear that $\nu({\bf p})\leq1.$ Choose a projection ${\bf q}\geq{\bf p}.$ Without loss of generality, ${\bf q}$ commutes with $Vx$ and ${\bf q}\leq e_{[\mu(1,Vx),\infty)}(|Vx|).$ Consider the operator
$$W:z\to {\bf q}\cdot Vz\cdot {\bf q},\quad z\in (L_q+L_{\infty})(\mathcal{M}).$$
By the assumption, $W:L_p(\mathcal{M})\to L_p({\bf q}\mathcal{N}{\bf q})$ and $W:L_{\infty}(\mathcal{M})\to L_{\infty}({\bf q}\mathcal{N}{\bf q})$ are contractions. By Lemma \ref{beta interpolation lemma}, we have
$$\nu(\Phi(|Wz|))\lesssim_{\Phi}\tau(\Phi(|z|)).$$
Setting $z=x,$ we conclude that
$$\int_0^1\Phi(\mu(t,Vx))dt=\tau(\Phi(|Wx|))\lesssim_{\Phi}\tau(\Phi(|x|)).$$
Let $\Phi_0(t)=\Phi(t^{\frac1p})$, $t\geq0.$ It follows from the assumption of $p$-convexity of $\Phi$ that $\Phi_0$ is convex. For an arbitrary $z\in L_p(\mathcal{M}),$ by the Jensen inequality we have
$$\Phi(\|z\|_p)=\Phi_0\Big(\int_0^1\mu^p(t,z)dt\Big)\leq\int_0^1\Phi_0\Big(\mu^p(t,z)\Big)dt=\tau\big(\Phi(|z|)\big).$$
It follows from \eqref{p+q} and the assumption of $V$ that
$$\Phi(\|Vx\|_{L_1+L_2})\lesssim_\Phi\Phi(\|Vx\|_{L_p+L_2})\lesssim\Phi(\|x\|_p)\leq\tau\big(\Phi(|x|)\big).$$
This concludes the proof.
\end{proof}

We now provide the proof of Theorem \ref{modular thm}.
\begin{proof}[Proof of Theorem \ref{modular thm}] First, let $x_k,$ $k\geq0,$ be mean zero independent random variables. We first prove \eqref{modsym}. Recall that $T$ is the operator introduced in the proof of Theorem \ref{js thm}. By Lemma \ref{fifth interpolation lemma}, we have
$$\tau\big(\Phi(|Tx|)\big)\lesssim_{\Phi}\int_0^1\Phi(\mu(t,x))dt+\Phi(\|x\|_{L_1+L_2}).$$
Applying the latter inequality to the $x=\sum_{k\geq0}x_k\otimes e_k,$ we obtain
$$\tau\big(\Phi(|\sum_{k\geq0}x_k|)\big)\lesssim_{\Phi}\mathbb E(\Phi(\mu(X)\chi_{(0,1)}))+\Phi(\|X\|_{L_1+L_2}).$$
By Lemma \ref{sixth interpolation lemma}, we have
$$\int_0^1\Phi(\mu(t,T^*x))dt+\Phi(\|T^*x\|_{L_1+L_2})\lesssim_{\Phi}\tau\big(\Phi(|x|)\big).$$
Applying the latter inequality to the $x=\sum_{k\geq0}x_k,$ we obtain
$$\mathbb E(\Phi(\mu(X)\chi_{(0,1)}))+\Phi(\|X\|_{L_1+L_2})\lesssim_\Phi\tau\big(\Phi(|\sum_{k\geq0}x_k|)\big).$$
This concludes the proof of \eqref{modsym}.

We now prove \eqref{modpos}. It follows from $q$-concavity of $\Phi$ that
$$\tau\big(\Phi(\sum_{k\geq0}x_k)\big)\lesssim_{\Phi}\tau\Big(\Phi\big(|\sum_{k\geq0}x_k-\tau(x_k)|\big)\Big)+\Phi(\|X\|_{L_1}).$$
Setting $Y=\sum_{k\geq0}(x_k-\tau(x_k))\otimes e_k$ and applying \eqref{modsym}, we arrive at
$$\tau\big(\Phi(\sum_{k\geq0}x_k)\big)\lesssim_{\Phi}\mathbb E(\Phi(\mu(Y)\chi_{(0,1)}))+\Phi(\|Y\|_{L_1+L_2})+\Phi(\|X\|_{L_1}).$$
Clearly, $Y\prec\prec 2X$ and, therefore, $\mu(Y)\chi_{(0,1)}\prec\prec 2\mu(X)\chi_{(0,1)}.$ Again using $q$-concavity, we obtain
$$\tau\big(\Phi(\sum_{k\geq0}x_k)\big)\lesssim_{\Phi}\mathbb E(\Phi(\mu(X)\chi_{(0,1)}))+\Phi(\|X\|_1).$$

To see the converse inequality, we can assume without loss of generality that there are only finitely many summands. Let $X=\sum_{k\geq0}x_k\otimes e_k.$ Set $x_{1k}=x_ke_{(\mu(1,X),\infty)}(x_k).$ We have
$$\mu(X)\chi_{(0,1)}\leq\mu\Big(\sum_{k\geq0}x_{1k}\otimes e_k\Big)+\mu(1,X)\chi_{(0,1)}.$$
It follows from \cite[Lemma 3.3.7]{LSZ} that
$$\sum_{k\geq0}x_{1k}\otimes e_k\prec\prec \sum_{k\geq0}x_{1k}\leq \sum_{k\geq0}x_k.$$
It follows from the convexity of $\Phi$ and \eqref{majorization phi} that
$$(\tau\otimes\Sigma)\Big(\Phi(\sum_{k\geq0}x_{1k}\otimes e_k)\Big)\leq\tau\Big(\Phi(\sum_{k\geq0}x_{1k})\Big)\leq\tau\Big(\Phi(\sum_{k\geq0}x_k)\Big).$$
It follows from $q$-concavity of $\Phi$ that
$$\mathbb E(\Phi(\mu(X)\chi_{(0,1)}))\lesssim (\tau\otimes\Sigma)\Big(\Phi(\sum_{k\geq0}x_{1k}\otimes e_k)\Big)+\Phi(\mu(1,X))\leq \tau\Big(\Phi(\sum_{k\geq0}x_k)\Big)+\Phi(\|X\|_{L_1}).$$
By Jensen inequality, we have
$$\Phi(\|X\|_{L_1})=\Phi(\|\sum_{k\geq0}x_k\|_{L_1(\M)})\leq\tau\Big(\Phi(\sum_{k\geq0}x_k)\Big).$$
A combination of the $2$ preceding inequalities yields
$$\mathbb E(\Phi(\mu(X)\chi_{(0,1)}))+\Phi(\|X\|_1)\lesssim_{\Phi}\tau\big(\Phi(\sum_{k\geq0}x_k)\big).$$
We conclude the proof of \eqref{modpos}.

We now prove \eqref{modkh}. Let $S$ be the operator introduced in the proof of Theorem \ref{khinchine}. By Lemma \ref{alpha interpolation lemma}, we have
$$(\tau\otimes{\rm Tr})\big(\Phi(|Sz|)\big)\lesssim_{\Phi}\tau(\Phi(|z|)).$$
Applying the operator $S$ to the element $\sum_{k\geq0}x_k,$ we arrive at
$$\tau\Big(\Phi\Big(\big(\sum_{k\geq0}x_k^2\big)^{\frac12}\Big)\Big)\lesssim_\Phi\tau\big(\Phi(|\sum_{k\geq0}x_k|)\big).$$
By Lemma \ref{alpha interpolation lemma}, we have
$$\tau\big(\Phi(|S^*z|)\big)\lesssim_\Phi\tau(\Phi(|z|)).$$
Applying the operator $S^*$ to the element $\sum_{k\geq0}x_k\otimes e_k,$ we arrive at
$$\tau\big(\Phi(|\sum_{k\geq0}x_k|)\big)\lesssim_\Phi\tau\Big(\Phi\Big(\big(\sum_{k\geq0}x_k^2\big)^{\frac12}\Big)\Big).$$
This concludes the proof of \eqref{modkh}. Hence, the proof of Theorem \ref{modular thm} is complete.
\end{proof}

\begin{rmk} It follows from the proof of Theorem \ref{modular thm} \eqref{modkh} that if $\{x_k\}_{k\geq0}$ is an arbitrary sequence of mean zero independent elements (not necessarily self-adjoint) from $L_\Phi(\M)$ where $\Phi$ is a $p$-convex and $q$-concave Orlicz function with $1<p\leq q<\infty,$ then
$$\tau\Big(\Phi\Big(\big(\sum_{k\geq0}|x_k|^2\big)^{\frac12}\Big)\Big)\approx_\Phi\tau\big(\Phi(|\sum_{k\geq0}x_k|)\big)
\approx_\Phi\tau\Big(\Phi\Big(\big(\sum_{k\geq0}|x_k^*|^2\big)^{\frac12}\Big)\Big).$$
Indeed, the first equality is proved by verbatim repetition of the argument, while the second one is proved by substituting $x_k^*$ instead of $x_k.$
\end{rmk}

\end{document}